%% file: arxiv.tex
\documentclass[11pt,onecolumn]{article}

\usepackage{a4wide}
\usepackage{amsmath,amsthm,epsfig,amssymb,amsbsy}
\usepackage{enumerate}
\usepackage{comment}
\usepackage{algorithm}
\usepackage{algorithmic}
\usepackage{amsfonts}       
\usepackage{dsfont}
\usepackage{enumitem}
\usepackage{amssymb}
\usepackage{comment}
\usepackage{adjustbox}

\usepackage{pgfplots}
\pgfplotsset{compat=1.18}
\usepackage{subcaption}

\usepackage{pgfplotstable}
\usepackage{multirow}
\usepackage{graphicx}

\usepackage[
giveninits=true,
maxbibnames=9,
maxcitenames=2,
backend=biber,
bibstyle=numeric,
doi=false,isbn=false,url=false,
]{biblatex}
\renewbibmacro{in:}{
    \ifentrytype{article}
    {}
    {\printtext{\bibstring{in}\intitlepunct}}}
\newenvironment{keywords}{
\begin{paragraph}{Keywords:}
}
{
\end{paragraph}
}
    
\newenvironment{subclass}{
\begin{paragraph}{AMS Subject Classification:}
}
{\end{paragraph}
}

\DeclareMathOperator*{\argmin}{arg\,min}
\DeclareMathOperator*{\argminimax}{arg\,minimax}
\DeclareMathOperator*{\argmax}{arg\,max}

\DeclareMathOperator{\infconv}{\mathbin{\square}}

\DeclareMathOperator{\dom}{dom}

\DeclareMathOperator{\intr}{int}

\DeclareMathOperator{\id}{id}

\newcommand{\bR}{\mathbb{R}}

\newcommand{\bN}{\mathbb{N}}

\newcommand{\exR}{\overline{\mathbb{R}}}

\makeatletter
\newcommand{\aprox}[3][\@nil]{%
  \def\tmp{#1}%
   \ifx\tmp\@nnil
       \operatorname{prox}_{#3}^{#2}
    \else
         \operatorname{prox}_{#3}^{#1 \star #2}
    \fi}

\newcommand{\aenv}[3][\@nil]{%
  \def\tmp{#1}%
   \ifx\tmp\@nnil
       \operatorname{env}_{#3}^{#2}
    \else
         \operatorname{env}_{#3}^{#1 \star #2}
    \fi}

\newcommand{\bprox}[3][\@nil]{%
  \def\tmp{#1}%
   \ifx\tmp\@nnil
       \operatorname{bprox}_{#3}^{#2}
    \else
        \operatorname{bprox}_{#3}^{#1 #2}
    \fi}
\makeatother

\usepackage{hyperref}
\hypersetup{
    colorlinks=true,
    linkcolor=blue,
    filecolor=magenta,
    citecolor =magenta,  
    urlcolor=magenta,
    pdftitle={Power PPA and ALM}
    }
\usepackage[capitalize]{cleveref}[0.19]

\crefname{section}{section}{sections}
\crefname{subsection}{subsection}{subsections}
\Crefname{section}{Section}{Sections}
\Crefname{subsection}{Subsection}{Subsections}

\Crefname{figure}{Figure}{Figures}

\crefformat{equation}{\textup{#2(#1)#3}}
\crefrangeformat{equation}{\textup{#3(#1)#4--#5(#2)#6}}
\crefmultiformat{equation}{\textup{#2(#1)#3}}{ and \textup{#2(#1)#3}}
{, \textup{#2(#1)#3}}{, and \textup{#2(#1)#3}}
\crefrangemultiformat{equation}{\textup{#3(#1)#4--#5(#2)#6}}%
{ and \textup{#3(#1)#4--#5(#2)#6}}{, \textup{#3(#1)#4--#5(#2)#6}}{, and \textup{#3(#1)#4--#5(#2)#6}}

\Crefformat{equation}{#2Equation~\textup{(#1)}#3}
\Crefrangeformat{equation}{Equations~\textup{#3(#1)#4--#5(#2)#6}}
\Crefmultiformat{equation}{Equations~\textup{#2(#1)#3}}{ and \textup{#2(#1)#3}}
{, \textup{#2(#1)#3}}{, and \textup{#2(#1)#3}}
\Crefrangemultiformat{equation}{Equations~\textup{#3(#1)#4--#5(#2)#6}}%
{ and \textup{#3(#1)#4--#5(#2)#6}}{, \textup{#3(#1)#4--#5(#2)#6}}{, and \textup{#3(#1)#4--#5(#2)#6}}

\usepackage{crossreftools}
\newtheorem{theorem}{Theorem}[section]
\newlist{thmenum}{enumerate}{1} 
\setlist[thmenum]{label=(\roman*), ref=\theproposition(\roman*), font=\rm}
\crefalias{thmenumi}{theorem}

\newtheorem{lemma}[theorem]{Lemma}
\newlist{lemenum}{enumerate}{1} 
\setlist[lemenum]{label=(\roman*), ref=\theproposition(\roman*), font=\rm}
\crefalias{lemenumi}{lemma} 

\newlist{assumenum}{enumerate}{1} 
\setlist[assumenum]{leftmargin=2.1cm,label=(A\arabic*),font=\bfseries}
\creflabelformat{assumenumi}{#2#1#3}
\crefname{assumenumi}{assumption}{assumptions}
\Crefname{assumenumi}{Assumption}{Assumptions}

\newtheorem{proposition}[theorem]{Proposition}
\newlist{propenum}{enumerate}{1} 
\setlist[propenum]{label=(\roman*), ref=\theproposition(\roman*), font=\rm}
\crefalias{propenumi}{proposition} 

\newtheorem{definition}[theorem]{Definition}
\newlist{defenum}{enumerate}{1} 
\setlist[defenum]{label=(\roman*), ref=\theproposition(\roman*), font=\rm}
\crefalias{defenumi}{definition} 

\theoremstyle{remark}

\newtheorem{example}[theorem]{Example}

\def\ifsvjour{false}

\numberwithin{equation}{section}

\title{Global Convergence Analysis of the Power Proximal Point and Augmented Lagrangian Method}
\author{Konstantinos A. Oikonomidis\thanks{KU Leuven,
		Department of Electrical Engineering (ESAT-STADIUS),
		Kasteelpark Arenberg 10, 3001 Leuven, Belgium~
		{\tt%
			\href{mailto:konstantinos.oikonomidis@kuleuven.be}{\{konstantinos.oikonomidis,}%
            \href{mailto:alexander.bodard@kuleuven.be}{alexander.bodard,}%
			\href{mailto:panos.patrinos@kuleuven.be}{panos.patrinos\}}%
			\href{mailto:konstantinos.oikonomidis@kuleuven.be,alexander.bodard@kuleuven.be,panos.patrinos@kuleuven.be}{@kuleuven.be}%
		}
	} \and 
Alexander Bodard\footnotemark[1] \and Emanuel Laude\thanks{Proxima Fusion GmbH, Fl\"oßergasse 2, 81369 Munich, Germany \texttt{ \href{mailto:elaude@proximafusion.com}{elaude@proximafusion.com}} \\
This work was supported by the Research Foundation Flanders (FWO) research projects G081222N, G033822N, G0A0920N; the Research Council KU Leuven C1 project with ID C14/24/103.} \and Panagiotis Patrinos\footnotemark[1]}

\begin{document}

\maketitle
\begin{abstract}
\input{abstract}
\end{abstract}
\begin{subclass}
65K05 $\cdot$ 49J52 $\cdot$ 90C30
\end{subclass}

\tableofcontents

\input{content}

\newpage
\appendix  

\section{Additional tables} \label{sec-append:additional-tables}
\input{additional_tables}

\section{Missing proofs} \label{appendix:missing-proofs}

\subsection{Proof of \cref{thm:errors_alm}} \label{sec-append:prf-errors_alm}
\input{proofs/alm}


\subsection{Proof of \cref{thm:alm-linear}} \label{sec-append:prf-alm-linear}
\input{proofs/global_linear}

\newpage
\printbibliography

\end{document}

%% file: abstract.tex
In this paper we study an unconventional inexact \emph{Augmented Lagrangian Method} (ALM) for convex optimization problems, as first proposed by Bertsekas, wherein the penalty term is a potentially non-Euclidean norm raised to a power between one and two. We analyze the algorithm through the lens of a nonlinear \emph{Proximal Point Method} (PPM), as originally introduced by Luque, applied to the dual problem. While Luque analyzes the order of local convergence of the scheme with Euclidean norms our focus is on the non-Euclidean case which prevents us from using standard tools for the analysis such as the nonexpansiveness of the proximal mapping.
To allow for errors in the primal update, we derive two implementable stopping criteria under which we analyze both the global and the local convergence rates of the algorithm. More specifically, we show that the method enjoys a fast sublinear global rate in general and a local superlinear rate under suitable growth assumptions. We also highlight that the power ALM can be interpreted as classical ALM with an implicitly defined penalty-parameter schedule, reducing its parameter dependence.
Our experiments on a number of relevant problems suggest that for certain powers the method performs similarly to a classical ALM with fine-tuned adaptive penalty rule, despite involving fewer parameters.

%% file: content.tex
\def\true{true}
\def\false{false}

\section{Introduction}
The \emph{Augmented Lagrangian Method} (ALM) introduced in 1969 by Hestenes and Powell \cite{hestenes1969multiplier,powell1969method} is a popular algorithm that allows one to cast a difficult nonsmooth or constrained problem into a sequence of smooth ones. Rockafellar showed in the seventies that, in the convex regime, the augmented Lagrangian method can be interpreted in terms of the \emph{Proximal Point Method} (PPM) applied to the Lagrangian dual problem \cite{rockafellar1976augmented}. However, the classical PPM requires careful selection of the step-sizes to attain fast convergence to a feasible solution \cite{luque1984asymptotic}.
In this work we consider a variant of the augmented Lagrangian method for convex optimization problems that can be interpreted as a classical augmented Lagrangian method with an implicitly defined adaptive penalty scheme.
For that purpose we consider a different proximal point method that is obtained by replacing the squared Euclidean proximal term (prox-function) with a higher-power of a potentially non-Euclidean norm. We refer to this variant of the PPM as the \emph{power} PPM. Specializing to the Euclidean norm raised to some power this is an instance of Luque's nonlinear proximal point method \cite{luque1984nonlinear,luque1986nonlinear,luque1987nonlinear}. Its dual counterpart, the nonlinear augmented Lagrangian method was analyzed by the same author \cite{luque1986nonlinearalm} expanding upon the work of Bertsekas \cite{bertsekas2014constrained}, who studied the algorithm without its connection to the proximal point method.
However, both works are focused on the local convergence of the method only.
More recently, Nesterov rediscovered Luque's nonlinear proximal point method in the context of tensor methods \cite{nesterov2021inexact}, analyzing its fast global sublinear convergence rate, see also \cite{ahookhosh2024high,ahookhosh_high-order_2021}.
However, the results in \cite{nesterov2021inexact} are specialized to Euclidean norms and the error tolerances are incorporated in terms of a tailored relative error bound that is impractical in the context of ALM. 
We note that penalty and augmented Lagrangian methods where the penalty term is raised to some power smaller than $2$, similarly to the proposed ALM, have also been analyzed recently \cite{izmailov2023convergence,bourkhissi2025convergenceanalysislinearizedellq,bodard2024inexactpoweraugmentedlagrangian}.

Complementing the results of \cite{nesterov2021inexact,luque1986nonlinearalm,luque1987nonlinear}, we analyze the global convergence rate and the order of local convergence of the power proximal point method and its dual counterpart, the power augmented Lagrangian method (power ALM), under two implementable stopping criteria. Particular emphasis is put on the usage of non-Euclidean norms which goes beyond the existing setup in \cite{luque1987nonlinear}, and comes with the challenge of a lack of nonexpansiveness \cite{laude2025anisotropic}. As a remedy, we leverage the objective function or its anisotropic Moreau envelope \cite{laude2025anisotropic} as merit functions that replace the more classical distance to the solution.

The favorable global rates and the higher order of local convergence is traded off with an increased difficulty of the subproblems which are H\"older continuous in general. In our simulations we therefore utilize a recent family of solvers \cite{nesterov2015universal, li2024simpleuniformlyoptimalmethod, oikonomidis2024adaptive} tailored for (locally) H\"older continous minimization problems besides classical off-the-shelf implementations of (L-)BFGS. Our numerical evaluation shows that the combined scheme behaves favorably when compared to the classical ALM with a fixed penalty parameter while it performs on par with classical ALM with a fine-tuned penalty parameter schedule. This suggests to interpret power ALM as a penalty-parameter-schedule-free augmented Lagrangian method.

More precisely, our contributions are summarized as follows:

\begin{itemize}
    \item Complementing Nesterov \cite{nesterov2021inexact} we derive fast sublinear rates for the suboptimality gap under more suitable (in the context of the ALM) stopping criteria based on either (1) $\varepsilon$-subgradients or (2) $\varepsilon$-proximal points. A complication arises in the latter case (2) whenever the cost is extended-valued due to infeasibility of iterates and renders the suboptimality gap unsuited for the analysis. This is remedied by interpreting the power PPM as an \emph{anisotropic gradient descent method} \cite{laude2025anisotropic} applied to the anisotropic Moreau envelope which distinguishes our approach from existing ones. In addition, we show that the suboptimality gap converges locally with higher order under a suitable growth condition.
    \item Specializing to the augmented Lagrangian setup we show that our stopping criteria are implementable for the power ALM. In addition we show that the ergodic primal-dual gap converges with the same fast rate as the non-ergodic dual suboptimality. We show that for linear equality constraints this translates to fast global ergodic rates for infeasibility and the primal suboptimality gaps.
    \item We compare the performance of the classical and power ALM on a collection of relevant problems using various powers $p$ and a collection of inner solvers that perform well on H\"older continuous minimization problems. 
    We show that for certain choices of $p$ the power ALM consistently outperforms its classical counterpart with a fixed penalty parameter in terms of the cumulative number of inner iterations. Moreover, it behaves similarly to classical ALM with a fine-tuned adaptive schedule for the penalty parameter despite involving fewer parameters.
\end{itemize}

The remainder of the paper is organized as follows:

In \cref{section:inexact_pppa}, we analyze the global and local convergence of the power proximal point method for nonsmooth convex minimization problems where the prox-function is either the $\ell_2$-norm or the $\ell_{p+1}$-norm raised to the power $p+1$. In \cref{section:ALM}, we specialize our results to the augmented Lagrangian setup and analyze the ergodic convergence of the power ALM. \Cref{section:numerics} presents the numerical evaluation of the method. \Cref{section:conclusion} concludes the paper.
\subsection*{Notation}
We denote by $\langle \cdot, \cdot\rangle$ the standard Euclidean inner product and by $\bR^n \times \bR^m$ the Cartesian product between $\bR^n$ and $\bR^m$, while $\bN$ is the set of the natural numbers and $\bN_0 = \bN \cup \{0\}$. The effective domain of an extended real-valued function $f : \bR^n \to \exR :=\bR \cup \{+\infty\}$ is denoted by $\dom f:=\{x\in\bR^n : f(x)<\infty\}$, and we say that $f$ is proper if $\dom f\neq\emptyset$; lower semicontinuous (lsc) if $f(\bar x)\leq\liminf_{x\to\bar x}f(x)$ for all $\bar x\in\bR^n$. We define by $\Gamma_0(\bR^n)$ the class of all proper, lsc convex functions $f:\bR^n \to \exR$ and with $\mathcal{C}^k(\bR^n)$ the ones which are $k$ times continuously differentiable on $\bR^n$.
For any functions $f :\bR^n \to \exR$ and $g:\bR^n \to \exR$ we define the infimal convolution or epi-addition of $f$ and $g$ as $(f \infconv g)(x) = \inf_{y \in \bR^n} g(x-y) + f(y)$.
We denote by $\bR^n_+=[0, +\infty)^n$ the nonnegative orthant. Otherwise we adopt the notation from \cite{RoWe98}. 

%

\section{Power proximal point and augmented Lagrangian method} \label{sec:power_alm}
This section first introduces the class of problems considered in this work, and formulates both the Langrangian primal and dual problems.
Then, the power augmented Lagrangian method is derived as an instance of the power proximal point method applied to the Lagrangian dual problem, generalizing the Euclidean setup \cite{rockafellar1974augmented}.
In addition, connections to generalized and sharp Lagrangians \cite{RoWe98} are provided. 

\subsection{Problem class and duality}
In this work we consider convex convex-composite problems that take the form
\begin{align}\label{eq:problem}
\inf_{x \in \bR^n}\big\{\varphi(x)\equiv f(x) + g(\mathcal{A}(x)) \big\},
\end{align}
where $f\in \Gamma_0(\bR^n)$ and $g \in \Gamma_0(\bR^m)$ and $\mathcal{A} :\bR^n \to \bR^m$ is a nonlinear mapping.
In the Rockafellian perturbartion framework for convex duality \cite{RoWe98} the primal problem is embedded in a family of parametric minimization problems of the form $\inf_{x \in \bR^n} F(x, u)$ where $F :\bR^n \times \bR^m \to \exR$ is defined by $F(x,u):=f(x) + g(\mathcal{A}(x) +u)$. Then, \cref{eq:problem} is recovered by the choice $u=0$ as $\inf_{x \in \bR^n} F(x, 0)$. By dualizing the linear equality constraint $u=0$ with a Lagrange multiplier $y$ we obtain the convex-concave Lagrangian $L : \bR^n \times \bR^m \to \exR$ of \cref{eq:problem}:
\begin{align} 
L(x,y) 
&=f(x) + \langle A(x), y\rangle - g^*(y).
\end{align}
The Lagrange dual function is given as
\begin{align} \label{eq:dual_function}
\varrho(y):=\inf_{x \in \bR^n}L(x,y),
\end{align}
and the Lagrange dual problem can be written as
\begin{align} \label{eq:dual_problem}
\sup_{y \in \bR^m}~\varrho(y) = \sup_{y \in \bR^m}\inf_{x \in \bR^n}L(x,y).
\end{align}
For the remainder of this paper we assume that strong duality, i.e., the identity $
    \sup_{y \in \bR^m}\varrho(y) = \inf_{x \in\bR^n} \varphi(x)$,
holds and primal and dual solutions exist. Sufficient conditions for convexity of $\varphi$ and strong duality for nonlinear $\mathcal{A}$ in a general conic optimization context are explored in \cite{bonnans2013perturbation,burke2021study}. A classical example is when $g$ is the indicator function of the nonpositive orthant, and the component-functions of $\mathcal{A}$ are convex and smooth, or when $g$ is the indicator of a closed, convex cone and $\mathcal{A}$ is affine. A classical sufficient condition for strong duality in this case is obtained via Slater's condition.

\subsection{Dual power proximal point}
The algorithm that is presented in this paper is obtained by applying a nonlinear extension of the proximal point scheme to the Lagrangian dual problem \cref{eq:dual_problem}.
For notational convenience we introduce
\begin{align}
\phi:=\tfrac{1}{p+1}\|\cdot\|_r^{p+1},
\end{align}
for some $p\geq 1$, where $\|\cdot\|_r$ is either the Euclidean norm, i.e., $r=2$ or the $(p+1)$-norm, i.e., $r={p+1}$. For brevity we omit $r$ and write $\|\cdot\|_s=\|\cdot\|_*$, for $s \in \{2, q+1\}$ with $\frac{1}{p+1} + \frac{1}{q+1}=1$, for the dual norm in the remainder of the manuscript. Note the convex conjugacy relations $\phi^*=\tfrac{1}{q+1}\|\cdot\|_*^{q+1}$ and $\nabla \phi \circ \nabla \phi^* = \id$.
We highlight that in applications the $(p+1)$-norm can have certain advantages over the $2$-norm because it leads to a separable $\phi$ whereas $\tfrac{1}{p+1}\|\cdot\|_2^{p+1}$ is nonseparable unless $p=1$.
In addition, we introduce the epi-scaling $\lambda \star \phi$ of $\phi$, defined as $(\lambda \star \phi)(x) = 
        \lambda \phi(\lambda^{-1} x)$ if $\lambda > 0$ and $(\lambda \star \phi)(x)=\delta_{\{0\}}(x)$ otherwise,
whose convex conjugate amounts to $(\lambda \star \phi)^* = \lambda \phi^*.$
For our specific choice of $\phi$ we have $(\lambda \star \phi)(x) = \lambda^{-p}\tfrac{1}{p+1}\|x\|^{p+1}$ and $(\lambda \star \phi)^* = \lambda \tfrac{1}{q+1}\|\cdot\|_*^{q+1}$ with $1=\tfrac{1}{q+1}+\tfrac{1}{p+1}$.

The algorithm that is studied in this paper amounts to the proximal point method with prox-term $\phi$ and step-size $\lambda$ applied to the negative dual function $\psi := -\varrho$:
\begin{align} \label{eq:ppa}
y^{k+1} &= \aprox[\lambda]{\phi}{\psi}(y^k) = \argmax_{y \in \bR^m}~\varrho(y) - \lambda \star \phi(y^k - y).
\end{align}
Expanding the definition of the Lagrange dual function $\varrho$, the exact version of the algorithm can be rewritten:
\begin{align} \label{eq:update_minimax}
(x^{k+1},y^{k+1}) &= \argminimax_{x \in \bR^n, y \in \bR^m}~L(x,y) - \lambda \star \phi(y^k - y),
\end{align}
where $x^{k}$ is the primal variable.
We introduce the corresponding augmented Lagrangian with parameter $\lambda$
\begin{align} \label{eq:augmented_lagrangian}
L_\lambda (x,y) &:= \sup_{\eta \in \bR^m} L(x,\eta) - \lambda \star \phi(y-\eta).
\end{align}
Note that $-L_\lambda (x^{k+1}, \cdot)$ is the infimal convolution of $-L(x^{k+1},\cdot)$ and $\lambda \star \phi$ and thus, by \cite[Lemma 2.7(ii)]{laude2025anisotropic}, its gradient w.r.t. $y$ amounts to $-\nabla_y L_\lambda (x^{k+1},y^k) = \nabla \phi(\lambda^{-1}(y^k - y^{k+1}))$ for $y^{k+1} := \argmax_{y \in \bR^m} \; L(x^{k+1},y) - \lambda \star \phi(y^k - y)$.
Then \cref{eq:update_minimax} separates into a primal and a consecutive dual update:
\begin{align} 
x^{k+1} &= \argmin_{x\in \bR^n} \;L_\lambda (x,y^k) \label{eq:exact-power-alm-primal} \\ 
y^{k+1} &= y^k + \lambda \nabla \phi^*(\nabla_y L_\lambda (x^{k+1}, y^k)). \label{eq:exact-power-alm-dual}
\end{align}
While the dual update \cref{eq:exact-power-alm-dual} can typically be solved in closed form given $x^{k+1}$, the primal update \cref{eq:exact-power-alm-primal} is solved inexactly using an iterative algorithm, e.g., a Quasi-Newton method. Since the dual update depends on $x^{k+1}$ the error in the primal update propagates to the dual update. This requires us to analyze the algorithm under a certain error tolerance.
We conclude this section with a simple example for the augmented Lagrangian function.
\begin{example}[linear equality constraints] \label{example:linear_eq_constraints}
    Let $g=\delta_{\{0\}}$ and $\mathcal{A}$ be an affine linear mapping $\mathcal{A}(x)=Ax-b$ for matrix $A \in \bR^{m \times n}$ and $b \in \bR^m$. Then $g(\mathcal{A}(x))$ encodes the equality constraint $\mathcal{A}(x)=b$ and the augmented Lagrangian reads, since $(\lambda \star \phi)^* = \lambda \phi^*$ and $\phi(-x)=\phi(x)$:
    \begin{align} \label{eq:power_alm_lin_eq}
    L_\lambda (x,y) &=f(x) + \sup_{\eta \in \bR^m} \langle Ax-b,\eta \rangle  - \lambda \star \phi(y - \eta) \notag \\
    &=f(x)+\langle y, Ax-b\rangle + \lambda \phi^*(Ax-b),
    \end{align}
    where the penalty $\phi^*=\tfrac{1}{q+1}\|\cdot\|_*^{q+1}$ with $1=\tfrac{1}{q+1}+\tfrac{1}{p+1}$.
    The gradient of the augmented Lagrangian w.r.t. $y$ amounts to $\nabla_y L_\lambda(x,y)= Ax-b$ and hence the dual update comes out as
    \begin{align} \label{eq:dual_update_linear_equality}
        y^{k+1} = y^k + \lambda \nabla \phi^*(A x^{k+1}-b).
    \end{align}
\end{example}

Assume that $\lambda \star \phi=\tfrac{\lambda}{p+1}\|\lambda^{-1}(\cdot)\|^{p+1}$. Then for $p \to \infty$, $\lambda \star \phi$ converges (in the pointwise sense) to the indicator function $\delta_{\lambda B_{\|\cdot\|}}$ of the unit ball corresponding to the norm $\|\cdot\|$ scaled with radius $\lambda$. This limit case can be interpreted in terms of a trust-region constraint with radius $\lambda$ in the proximal point scheme \cref{eq:ppa}.
Dually, since $\phi^*=\tfrac{1}{q+1}\|\cdot\|_*^{q+1}$ with $1=\tfrac{1}{q+1}+\tfrac{1}{p+1}$, $q+1 \to 1$ and hence $\lambda \phi^*$ converges pointwisely to the dual norm $\lambda\|\cdot\|_*$ scaled with $\lambda$.
In this limit case, the function $\phi^*$ is a so-called exact penalty \cite{bertsekas1975necessary}, which guarantees that for $\lambda$ sufficiently large the constraints are feasible at the minimizer of the augmented Lagrangian for a certain multiplier $y^0$. Furthermore, in this case the minimizer of $L_\lambda(\cdot, y^0)$ coincides with the minimizer of the original constrained problem \cref{eq:problem}; see \cite[Definition 11.60]{RoWe98}. This reflects the dual perspective in which proximal point with a trust-region constraint converges in a single step to the solution if $y^0$ is close to $y^* \in \argmax \varrho$ and $\lambda$ is taken sufficiently large. 

\section{Convergence analysis of the power proximal point method under inexactness}
\label{section:inexact_pppa}
In this section we analyze the global and local convergence rates of the power proximal point method under certain error tolerances when applied to the problem of minimizing a convex, proper lsc function $\psi \in \Gamma_0(\bR^n)$,  i.e.,
\begin{align} \label{eq:pppa_problem}
    \inf_{y \in \bR^m} \psi(y).
\end{align}
The results provided in this section are independent of the usecase of the augmented Lagrangian method.
In particular we show that the algorithm converges globally with a fast sublinear rate and transitions locally to fast superlinear convergence if $\psi$ satisfies a mild growth condition.

Classically, the convergence of the proximal point method is derived using the nonexpansiveness of the proximal mapping which guarantees that the distance to a solution is nonincreasing. This is still valid for the power proximal point method if the prox-function is a higher power of the Euclidean norm \cite{luque1984nonlinear}. However, in this work our focus is on non-Euclidean norms which lack any form of nonexpansiveness. Hence, the distance to the solution cannot be used to show convergence. Instead we rely on the cost $\psi$ or its anisotropic Moreau envelope \cite{laude2025anisotropic} as merit functions which are in fact the quantities of interest in minimization problems. Specifically, the fast rates are derived relative to the suboptimality gaps.

\subsection{Fast global convergence under \texorpdfstring{$\varepsilon$}{ε}-subgradients}
In this subsection we analyze the global complexity of the power proximal point method under a certain error tolerance.
Note that by first-order optimality the update $y^{k+1} = \aprox[\lambda]{\phi}{\psi}(y^k)$ can be expressed in terms of finding a pair $(v^k, y^{k+1})$ that satisfies the following nonlinear system of inclusions,
\begin{subequations}
\label{eq:exact_alg_update_y_full}
\begin{align} 
v^k &\in \partial \psi(y^{k+1}) \\
y^{k+1} &= y^k - \lambda \nabla \phi^*(v^k). \label{eq:exact_alg_update_y}
\end{align}
\end{subequations}
To incorporate error tolerances we replace $\partial \psi(y^{k+1})$ with the $\varepsilon_k$-subdifferential of $\psi$ at $y^{k+1}$, which is defined as follows.
For $\varepsilon > 0$ the vector $v \in \bR^n$ is an $\varepsilon$-subgradient of $\psi$ at $\bar y \in \dom \psi$ if the following inequality holds true:
\begin{align} \label{eq:eps_subgradient}
    \psi(y) \geq \psi(\bar y) + \langle v, y- \bar y\rangle - \varepsilon,
\end{align}
for all $x \in \dom \psi$. The set of all $\varepsilon$-subgradients at $\bar y$ is called the $\varepsilon$-subdifferential of $\psi$ at $\bar y$ and is denoted by $\partial_\varepsilon \psi(\bar y)$. The algorithm that is analyzed in this section is listed in \cref{alg:inexact_pppa_averaging}, which possibly incorporates an average with an anchor point $y^0$.
\begin{algorithm}
\caption{Inexact power proximal point method with averaging}

\label{alg:inexact_pppa_averaging}
\begin{algorithmic}
\REQUIRE Choose $y^0 \in \bR^m$ and a sequence $\{\theta_k\}_{k \in \bN_0}$ with $\theta_k \in (0, 1]$.
\FORALL{$k=0, 1, \dots$}
   \STATE Set $w^k = \theta_k y^k + (1 - \theta_k)y^0$.
   \STATE Choose $\varepsilon_{k} \geq 0$ and solve
   \begin{subequations}
   \begin{align}
    \label{eq:inexact_ppa_v_anchor}v^k &\in \partial_{\varepsilon_k} \psi(y^{k+1}) \\
    \label{eq:inexact_ppa_y_anchor} y^{k+1} &= w^k - \lambda \nabla \phi^*(v^k).
    \end{align}
    \end{subequations}
\ENDFOR
\end{algorithmic}
\end{algorithm}
Since $\partial \psi(\bar y) \subseteq \partial_\varepsilon \psi(\bar y)$, \cref{alg:inexact_pppa_averaging} is a strict generalization of the exact version.
To facilitate our analysis we need the following lemma:
\begin{lemma}\label{thm:uniform_convexity}
The following result holds for $\phi=\tfrac{1}{p+1}\|\cdot\|^{p+1}$ with $p \geq 1$. For any $x,y \in \bR^m$ the following inequality holds:
\begin{align} \label{eq:uniform_convexity}
    \phi(x) \geq \phi(y) + \langle \nabla \phi(y), x-y \rangle + \tfrac{1}{2^{p-1}}\phi(x-y).
\end{align}
\end{lemma}
\begin{proof}
    In the case $\phi = \tfrac{1}{p+1}\|\cdot\|_2^{p+1}$ the result follows from \cite[Lemma 4.2.3]{nesterov2003introductory}. If $\phi = \tfrac{1}{p+1}\|\cdot\|_{p+1}^{p+1}$ is chosen, then the result easily follows by considering $\phi$ as a separable sum of $\tfrac{1}{p+1}|\cdot|^{p+1}$ and then applying \cref{eq:uniform_convexity} to each summand.
\ifx\ifsvjour\true
\qed
\fi
\end{proof}


We first prove the following lemma:
\begin{lemma} \label{thm:decrease_inexact_ppa}
Let $w \in \bR^m$. Let $(y^+, v)$ such that $v \in \partial_{\varepsilon} \psi(y^+)$ and $y^+ = w - \lambda \nabla \phi^*(v)$. Then we have for all $y \in \bR^m$:
\begin{align} \label{eq:uniform_convexity_upd}
\psi(y^+) \leq \psi(y) + \lambda \star \phi(w-y) - \tfrac{1}{2^{p-1}} \lambda \star \phi(y^+ -y) + \varepsilon
\end{align}
\end{lemma}
\input{proofs/lemma_decrease}

We define the sequences of coefficients $\{A_k\}_{k=0}^\infty$ and $\{a_{k}\}_{k=1}^\infty$ with $A_0:=0$, $A_k := k^{p+1}$ and $a_{k+1} := A_{k+1} - A_{k}$ for $k\geq 1$ and for the remainder of this subsection we assume that there exists a $y^\star \in \argmin \psi$. Using \cref{thm:uniform_convexity}, we can describe the global convergence of \cref{alg:inexact_pppa_averaging} for different choices of the sequence $\{\theta_k\}_{k \in \bN_0}$. On the one hand, by choosing $\theta_k := \tfrac{A_k}{A_{k+1}}$, we obtain a scheme in which the update is given in terms of a convex combination of the proximal point and an anchor point. In this case the convergence proof closely follows \cite[Theorem 5]{doikov2020inexact} for the inexact tensor method with averaging. 

On the other hand, when using $\theta_k := 1$ the update depends only on the previous variable. In this setup our analysis is similar to that of \cite[Theorem 1]{doikov2020inexact} for the Monotone Inexact Tensor Method I, \cite[Algorithm 1]{doikov2020inexact}. However, a complication arises due to the fact that the evaluation of $\psi$ is not necessarily cheap. For example, in the context of constrained optimization the Lagrangian dual function might not have a simple closed form and as such evaluating it at each iteration might become very expensive.
Consequently, a monotonic decrease can no longer be guaranteed and thus the analysis has to be adapted. For this reason we introduce an enlargement of the diameter of the initial level-set \cite[Equation 8]{doikov2020inexact}. 
Assuming that the sequence of errors $\varepsilon_k$ is summable we define the constant $\beta:=\sum_{k=0}^\infty \varepsilon_k < \infty$ and the diameter of the level-set
\begin{align}
D_\beta:= \sup \{ \|y^\star - y\| : \psi(y) \leq \psi(y^0) + \beta\}.
\end{align}
We assume that $D_\beta < \infty$ is finite, which happens to be the case when $\psi$ is level-bounded, which is equivalent to $0 \in \intr \dom \psi^*$ \cite[Theorem 11.8(c)]{RoWe98}.
The following theorem describes the global convergence of \cref{alg:inexact_pppa_averaging} for both choices of $\theta_k$. Note that the derived convergence rate for $\theta_k = 1$ is worse than the one for $\theta_k = \tfrac{A_k}{A_{k+1}}$ since it depends on the diameter of the initial level-set, which can be large. In practice, however, we observed the opposite behavior: The variant with $\theta_k=1$ performs better throughout our experiments even for $p=1$ which corresponds to the classical PPM.

\begin{theorem} \label{thm:inexact_ppa_rate}
    Choose the sequence of errors $\{\varepsilon_k\}_{k \in \bN_0}$ such that $\varepsilon_k = \tfrac{c}{(k+1)^{p+1}},$
    for some $c\geq 0$. Let $y^\star \in \argmin \psi$. Then, for the sequence $\{y^k\}_{k \in \bN}$ generated by \cref{alg:inexact_pppa_averaging} the following statements hold true:
    \begin{thmenum}
        \item \label{thm:inexact_ppa_rate:averaging} If $\theta_k := \tfrac{A_k}{A_{k+1}}$, then
        $
            \psi(y^k) - \psi(y^\star) \leq \frac{\lambda^{-p}(p+1)^p \|y^0 - y^\star\|^{p+1} + c}{k^p}.
        $
        \item \label{thm:inexact_ppa_rate:simple} If $\theta_k := 1$, then
        $
            \psi(y^k) - \psi(y^\star) \leq \frac{\lambda^{-p}(p+1)^{p}D_\beta^{p+1} + c}{k^p}.
        $
    \end{thmenum} 
\end{theorem}
\ifx\ifsvjour\true
A proof is provided in \cref{sec-appd:prf-inexact_ppa_rate_averaging}.
\fi
\ifx\ifsvjour\false
\input{proofs/global_subgradient_averaging}
\fi



\subsection{Fast global convergence under \texorpdfstring{$\varepsilon$}{ε}-proximal points}
\label{subsec:anisotropic_env}
The algorithm studied in the previous section incorporates error tolerances via the notion of $\varepsilon$-subgradients and relies on the cost function $\psi$ as a merit function which is potentially extended-valued. As a consequence, the next iterate has to stay within the domain of $\psi$ which can be restrictive in practice. In this section we therefore consider a more relaxed error tolerance that requires the next iterate $y^{k+1}$ to stay close to the true proximal point $\aprox{\lambda\star\phi}{\psi}(y^k)$, referred to as an $\varepsilon$-proximal point of $\psi$ at $y^k$:
\begin{align} \label{eq:env_error_cond}
    \lambda \star \phi(y^{k+1} - \aprox{\lambda\star\phi}{\psi}(y^k)) \leq \varepsilon_k,
\end{align}
where $\varepsilon_k \geq 0$ and $\sum_{k=0}^\infty \varepsilon_k < \infty$.
Note that this error criterion generalizes the classical inexactness conditions studied in \cite{luque1984asymptotic}. 
For that purpose we leverage the real-valued and smooth anisotropic Moreau envelope \cite{laude2021conjugate,laude2025anisotropic} as a merit function and interpret the power proximal point method in terms of the anisotropic gradient descent method \cite{laude2025anisotropic} applied to the merit function. 

We define the left anisotropic Moreau envelope adapted from \cite[Definition 3.7]{laude2025anisotropic} which is given as the value function of the power proximal point update \cref{eq:ppa}: For $\lambda > 0$ and $y \in \bR^m$ the \emph{left anisotropic Moreau envelope} or in short anisotropic envelope of $\psi$ with parameter $\lambda$ at $y$ is defined as
    \begin{align} \label{eq:anisotropic_envelope}
        \aenv{\lambda\star\phi}{\psi}(y) := \inf_{\eta \in \bR^m} \{\psi(\eta) + \lambda \star \phi(y-\eta)\}.
    \end{align}
Thanks to \cite[Lemma 3.9(ii)]{laude2025anisotropic} we have that
\begin{align} \label{eq:gradient_formula}
\nabla \aenv{\lambda\star\phi}{\psi}(y^k) = \nabla \phi(\lambda^{-1}(y^k - \aprox{\lambda\star\phi}{\psi}(y^k))).
\end{align}
Rearranging the identity reveals that the exact proximal point scheme \cref{eq:ppa} is equivalent to the anisotropic gradient descent method \cite[Algorithm 1]{laude2025anisotropic} applied to the anisotropic envelope of $\psi$, since:
\begin{align} \label{eq:gradient_descent_update}
    y^{k+1} = \aprox{\lambda\star\phi}{\psi}(y^k) =y^k - \lambda \nabla \phi^* (\nabla \aenv{\lambda\star\phi}{\psi}(y^k)).
\end{align}
To prove the main convergence result of this subsection we exploit \cite[Proposition 4.1]{laude2025anisotropic} which reveals that the anisotropic Moreau envelope $\aenv{\lambda\star\phi}{\psi}$ has the anisotropic descent property \cite[Definition 3.1]{laude2025anisotropic} with reference function $\phi$:
\begin{definition}
    Let $h \in \mathcal{C}(\bR^m)$ and $\phi=\tfrac{1}{p+1}\|\cdot\|^{p+1}$. We say that $h$ satisfies the anisotropic descent property relative to $\phi$ with constant $L$ if for all $\eta \in \bR^m$
    \begin{equation}\label{eq:adescent}
h(y) \leq h(\eta) + \tfrac{1}{L} \star \phi(y-\eta + L^{-1}\nabla \phi^*(\nabla h(\eta))) - \tfrac{1}{L} \star \phi(L^{-1}\nabla \phi^*(\nabla h(\eta))).
\end{equation}
\end{definition}
Note that whenever $p=1$, and $\|\cdot\|=\|\cdot\|_2$ we obtain the classical Euclidean descent lemma by expanding the square.
Based on the anisotropic descent inequality above we can prove the following theorem:
\begin{theorem} \label{lemma:env_convergence}
Let  $\{y^{k}\}_{k \in \bN_0}$ be any sequence that satisfies \cref{eq:env_error_cond}. Then the following statments hold true:
\begin{thmenum}
    \item \label{lemma:env_convergence:ineq} $
    \aenv{\lambda\star\phi}{\psi}(y^{k+1}) -\aenv{\lambda\star\phi}{\psi}(y^k)\leq-\lambda \star \phi(y^k - \aprox{\lambda\star\phi}{\psi}(y^k)) + \varepsilon_k.$
    \item \label{lemma:env_convergence:diff} If $\{\varepsilon_k\}_{k \in \bN_0}$ is summable with $\sum_{k=0}^\infty \varepsilon_k \leq \alpha$, for $\alpha \geq 0$, and the set of solutions $Y^\star = \argmin \psi$ is nonempty, then $\|y^{k+1}-y^k\| \to 0.$
    \item \label{lemma:env_convergence:convergence}If $Y^\star$ is bounded, then $\{y^k\}_{k \in \bN_0}$ remains bounded and all of its limit points lie in $Y^\star$. Moreover, if $\varepsilon_k = \tfrac{c}{(k+1)^{p+1}}$ for some $c \geq 0$, then for all $k \geq 1$
     \begin{align} 
        \aenv{\lambda\star\phi}{\psi}(y^k) - \aenv{\lambda\star\phi}{\psi}(y^\star) \leq \tfrac{\lambda^{-p} D_\beta^{p+1} (p^2+p)^{p} + c}{k^p}.
    \end{align}
\end{thmenum}

\end{theorem}
\ifx\ifsvjour\true
A proof is included in \cref{sec-appd:env_convergence}.
\fi
\ifx\ifsvjour\false
\input{proofs/global_prox}
\fi

\subsection{Local convergence with higher order under function growth conditions}
\label{subsec:local_rate}
In this  subsection we establish the higher order of local convergence of the power proximal point scheme under a mild growth condition of the objective $\psi$.

In the following, let $Y^\star$ denote the set of minimizers of $\psi$, which we assume is nonempty and compact, and $\psi^\star$ be the minimum value of $\psi$. Let $d(y, Y^\star) := \inf_{\eta \in Y^\star}\|y - \eta\|_r$ for $r \in \{2, p+1\}$ denote the distance of $y$ to $Y^\star$. For brevity we again omit $r$.
The following result captures the order and rate of convergence of the algorithm with exact updates. A similar result is provided in \cite[Proposition 5.20]{bertsekas2014constrained} for the Augmented Lagrangian method. 
\begin{theorem} \label{thm:superlinear_rate}
Let $\{y^{k}\}_{k \in \bN_0}$ be the sequence of iterates generated by \cref{alg:inexact_pppa_averaging} for $\theta_k = 1$, $\varepsilon_k =0$. Moreover assume that $Y^\star$ is compact and there exists a $\delta$-neighborhood $\mathcal{N}(Y^\star, \delta):=\{y : d(y,Y^\star) < \delta\}$ of $Y^\star$ such that $\psi(y) - \psi^\star  \geq  \mu d^\nu(y, Y^\star)$, with $\nu > 1$, for all $y \in \mathcal{N}(Y^\star, \delta)$. Let $p > \nu - 1$. Then, $\{d(y^k, Y^\star)\}_{k \in \bN_0}$ converges locally with order $\omega := \tfrac{p}{\nu - 1} > 1$ and rate $\alpha :=(\frac{1}{\lambda \mu^q})^\omega$, i.e., for $k$ sufficiently large we have
\begin{align} \label{eq:superlinear_rate_p_norm}
    \frac{d(y^{k+1}, Y^\star)}{d^\omega(y^k, Y^\star)} \leq \alpha.
\end{align}
\end{theorem}
\ifx\ifsvjour\true
A proof is included in \cref{sec-appd:prf-superlinear}.
\fi
\ifx\ifsvjour\false
\input{proofs/superlinear}
\fi

Although \cref{thm:superlinear_rate} describes the local convergence of the exact version of the algorithm, the extension to the inexact case is straightforward due to the equivalence of the norms in $\bR^m$. More specifically, one could follow the analysis of \cite{luque1984nonlinear} and further specify the inexactness condition \cref{eq:env_error_cond} as follows:
$
    \|y^{k+1}- \aprox{\lambda\star\phi}{\psi}(y^k))\| \leq \varepsilon_k \min\{1, \|y^{k+1}-y^k\|^t\},
$
for some $t > 0$ and thus obtain an order of convergence of at least $\min\{t, \tfrac{1}{q(\nu - 1)}\}$. It is important to note that for $\phi = \tfrac{1}{p+1}\|\cdot\|_2^{p+1}$ as in \cite{luque1984nonlinear}, it is not necessary to assume that $Y^\star$ is compact. To the best of our knowledge, however, such a condition cannot be avoided when dealing with $(p+1)$-norms.

\section{The inexact power augmented Lagrangian method}\label{section:ALM}
In this section we specialize the results of the inexact power proximal point method to the augmented Lagrangian setting as described in \cref{sec:power_alm}. In particular, we provide implementable stopping criteria for the primal update $x^{k+1}$ that guarantee that the dual update $y^{k+1}$ is within the error tolerance of the inexact power proximal point method when applied to the Lagrangian dual function. Furthermore, we show that the ergodic primal-dual gap vanishes at the same rate as the dual suboptimality gap established in the previous section and provide global convergence results for important quantities in the constrained optimization framework. 
\subsection{Inexactness criteria}
Recall that $\psi := -\varrho$ is the negative Lagrangian dual function $\varrho$ which is defined in \cref{eq:dual_function}.
This subsection describes stopping criteria for the primal update 
\begin{align} \label{eq:general_inexact_primal}
    x^+ \approx \argmin_{x \in \bR^n} L_\lambda(x, y),
\end{align}
which guarantee that the corresponding dual update
\begin{align} \label{eq:general_inexact_dual}
    y^+ = y + \lambda \nabla \phi^*(\nabla_y L_\lambda(x^+, y)). 
\end{align}
is within the error tolerance for the inexact power proximal point scheme applied to $\psi$.
For that purpose we prove the following result:
\begin{proposition} \label{thm:errors_alm}
The following properties hold for $x^+$ and $y^+$:
\begin{propenum}
    \item \label{thm:errors_alm:al} Let $x^+$ such that $L(x^+, y^+) \leq \inf L(\cdot, y^+) + \varepsilon$. Then $-\nabla_y L_\lambda(x^+, y) \in \partial_{\varepsilon} \psi(y^+)$.
    
    \item \label{thm:errors_prox_alm:bound_alm}
    Let $x^+$ such that $L_\lambda(x^+, y) \leq \inf L_\lambda(\cdot, y) + \varepsilon$. Then $y^+$ is an $2^{p-1}\varepsilon$-proximal point, i.e., 
    $
    \lambda \star \phi(y^+ - \aprox{\phi}{\psi}(y)) \leq 2^{p-1}\varepsilon.
    $
    
    \item \label{thm:errors_prox_alm:grad_alm} Assume that $\dom f$ is bounded. Let $e \in \partial_x L_\lambda(x^+, y)$ such that $\|e\| \leq \varepsilon/D$, where $D:=\sup_{x,y \in \dom f} \|x - y\| < \infty$ is the diameter of $\dom f$. Then the hypotheses of the first and second item are valid.
\end{propenum}
\end{proposition}
A proof of this result is included in \cref{sec-append:prf-errors_alm}.

\Cref{thm:errors_prox_alm:grad_alm} suggests to verify whether the (sub)gradient norm is sufficiently small, i.e., $\|\nabla_x L_\lambda(x^+, y)\| \leq \varepsilon / D$, which is an easily verifiable condition. 
\Cref{thm:errors_alm:al} then ensures that \(- \nabla_y L_{\lambda}(x^+, y)\) is an \(\varepsilon\)-subgradient of \(\psi\) at \(y^+\).
Likewise, according to \cref{thm:errors_prox_alm:bound_alm} this condition implies that \(y^+\) is an $2^{p-1}\varepsilon$-proximal point.
Therefore, we can analyze the proposed inexact ALM through its equivalence with the inexact power proximal point method (\cref{alg:inexact_pppa_averaging}) to obtain convergence guarantees.

Regarding the assumption on the boundedness of the domain of $f$ in \cref{thm:errors_prox_alm:grad_alm}, we remark that it is a standard assumption in the context of ALM, especially when obtaining convergence rates \cite{nedelcu2014computational,liu2019nonergodic,xu2021iteration,kong2023iteration,necoara2019complexity} as it allows for implementable stopping criteria.

Based on the above result we can now specialize \cref{alg:inexact_pppa_averaging} to the inexact power augmented Lagrangian method listed in \cref{alg:inexact_al_averaging}.
\begin{algorithm}
\caption{Inexact power augmented Lagrangian method with averaging}
\label{alg:inexact_al_averaging}
\begin{algorithmic}
\REQUIRE Choose $y^0 \in \bR^m, x^0 \in \bR^n$ and a sequence $\{\theta_k\}_{k \in \bN_0}$ with $\theta_k \in (0, 1]$.
\FORALL{$k=0, 1, \dots$}
   \STATE Set $w^k = \theta_k y^k + (1 - \theta_k)y^0$.
   \STATE Choose $\varepsilon_{k} > 0$ and solve
   \begin{subequations}
   \begin{align} \label{eq:averaging_primal}
x^{k+1} &\approx \argmin_{x \in \bR^n} L_\lambda(x, w^k) \\
\label{eq:averaging_dual}
y^{k+1} &=  w^k + \lambda \nabla \phi^*(\nabla_y L_\lambda(x^{k+1}, w^k)).
\end{align}
\end{subequations}
\ENDFOR
\end{algorithmic}
\end{algorithm}

\subsection{Convergence of the primal-dual gap}
In this subsection we refine the analysis from the previous section in the primal-dual setup. Specifically, we show that the primal-dual gap decreases at the same rate as the dual sub\-optimality-gap; see \cref{thm:inexact_ppa_rate}. For that purpose we define for any $K\geq 1$ the ergodic primal iterate:
\begin{align} \label{eq:ergodic_iterates}
\breve x^K = \tfrac{\sum_{k=1}^{K} a_k x^k}{\sum_{k=1}^{K} a_k}.
\end{align}
In the following theorem we show that the ergodic primal-dual gap decreases at the same rate as the negative dual function.
The proof is
along the lines of the analysis in \cite{yang2020catalyst}, albeit adapted to the higher-order proximal-point setting where in general nonexpansiveness is not present and as such standard techniques are not directly applicable.
\begin{theorem}\label{thm:primal-dual-gap}
    Choose the sequence of errors $\{\varepsilon_k\}_{k \in \bN_0}$ such that $
        \varepsilon_k = \tfrac{c}{(k+1)^{p+1}},$
    for some $c\geq 0$. Let $(x^k,y^k)_{k \in \bN_0}$ be the sequence of iterates generated by \cref{alg:inexact_al_averaging} with $\theta_k := \tfrac{A_k}{A_{k+1}}$ and termination criterion $L(x^{k+1}, y^{k+1}) \leq \inf L(\cdot, y^{k+1}) + \varepsilon_k$.
    Then we have for all $K \geq 1$ and any $x \in \bR^n$ and any $y \in \bR^m$:
\begin{align}
L(\breve x^{K}, y)-L(x, y^{K}) \leq \tfrac{c + \lambda \star \phi(y^0 -y) (p+1)^{p+1}}{K^{p}}.
\end{align}
\end{theorem}
\ifx\ifsvjour\true
The proof is included in \cref{sec-append:prf-primal-dual-gap}.
\fi
\ifx\ifsvjour\false
\input{proofs/primal_dual_gap}
\fi

Note that this result is not specific to the Lagrangian function as it holds for any convex-concave function, where the inexact higher-order proximal point method is applied to the concave part.

\subsection{Convergence for problems with linear equality constraints}
In this subsection we revisit \cref{example:linear_eq_constraints} specializing to the problem of minimizing a convex function $f$ subject to linear equality constraints. First, we provide results for the ergodic iterates, harnessing the analysis of the primal-dual gap that is presented in \cref{thm:primal-dual-gap}. 
\begin{theorem} \label{thm:alm-linear-ergodic}
    In the situation of \cref{example:linear_eq_constraints} let $\{\breve x^K\}_{K \in \bN}$ be the ergodic primal iterates of \cref{alg:inexact_al_averaging} as defined in \eqref{eq:ergodic_iterates} with $\theta_k = \tfrac{A_k}{A_{k+1}}$, $\varepsilon_k = \tfrac{c}{(k+1)^{p+1}}$ for some $c \geq 0$ and termination criterion $L(x^{k+1}, y^{k+1}) \leq \inf L(\cdot, y^{k+1}) + \varepsilon_k$ for all $k \in \bN_0$. Then, the following inequality holds for $K \geq 1$:
    \begin{align}
        \max\{\|A\breve x^K - b\|_*, |f(\breve x^{K}) - f(x^\star)|\} \leq \tfrac{c + \max_{y \in \delta \mathcal{B}}\lambda \star \phi(y^0-y)(p+1)^{p+1}}{K^{p}},
    \end{align}
    where $\delta = 2\|y^\star\| + 1$, $\mathcal{B}:=\{y \in \bR^m: \|y\| \leq 1\}$ denotes the closed unit ball, and $(x^\star, y^\star)$ is a saddle-point of the problem.
\end{theorem}
\ifx\ifsvjour\true
The proof of this theorem is presented in \cref{sec-append:prf-alm-linear-ergodic}.
\fi
\ifx\ifsvjour\false
\input{proofs/ergodic_primal}
\fi

Moving on to the analysis of the nonergodic iterates, the following theorem compiles the corresponding results.
\begin{theorem} \label{thm:alm-linear}
    In the situation of \cref{example:linear_eq_constraints} let $\{x^k\}_{k \in \bN_0}$ be the primal iterates of \cref{alg:inexact_al_averaging} with $\theta_k = 1$, $\varepsilon_k = \tfrac{c}{(k+1)^{p+1}}$ and termination criterion $L(x^{k+1}, y^{k+1}) \leq \inf L(\cdot, y^{k+1}) + \varepsilon_k$ for all $k \in \bN_0$. Then, the following hold for $k \geq 1$:
    \begin{thmenum}
        \item 
        \label{thm:linear_eq:feasibility}
        $
            \|Ax^k - b\|_* = \mathcal{O}\bigg(k^{-\tfrac{p^2}{p+1}}\bigg)
        $
        \item 
        $
            |f(x^k)-f(x^*)| = \mathcal{O}\bigg(k^{-\tfrac{p^2}{p+1}}\bigg)
        $
        \label{thm:linear_eq:f_bound}.
    \end{thmenum}
\end{theorem}
The proof is presented in \cref{sec-append:prf-alm-linear}.

Notice that by choosing $p=1$, we obtain the rates described in \cite[Corollary 4.4]{liu2019nonergodic}. In this sense, it can be considered as a generalization of that result in the power ALM setting. Moreover, such rates can be shown, albeit with different constants, regardless of the specific algorithm. 

\subsection{Solution of the subproblems of the power augmented Lagrangian method}
Again we specialize to the minimization of a convex cost subject to linear equality constraints as in \cref{example:linear_eq_constraints}.
Although the convergence results of the power ALM in terms of outer iterations seem quite attractive when using $p > 1$, the inner subproblems become harder to solve. 
Notice that when choosing $\phi = \tfrac{1}{p+1}\|\cdot\|_2^{p+1}$ the term $\|Ax-b\|_2^{q+1}$ in \cref{eq:power_alm_lin_eq} is H\"older smooth of order $q$. In general, such a term slows down the convergence of first-order methods as shown in \cite{grimmer2023optimal}. More specifically, if $f$ is an $L$-smooth function, the augmented Lagrangian \cref{eq:power_alm_lin_eq} as a function in $x$ is comprised of functions with varying order of H\"older smoothness. This stands in contrast to the classical ALM, where the augmented Lagrangian is Lipschitz smooth and hence classical first-order methods such as the accelerated proximal gradient method can be applied. Nevertheless, again from \cite{grimmer2023optimal} we know that the inexact version of this algorithm, specifically the one presented in \cite{nesterov2015universal}, can be used to minimize \cref{eq:power_alm_lin_eq} as well. Moreover in a recent line of work, adaptive proximal gradient methods were introduced that can also tackle composite convex locally H\"older smooth problems \cite{oikonomidis2024adaptive, li2024simpleuniformlyoptimalmethod}.

\subsection{Interpretation as an implicit adaptive penalty scheme} \label{sec:implicit-adaptive-penalty-interpretation}

Since power ALM, \cref{alg:inexact_al_averaging} with $\theta_k = 1$, can be interpreted as a proximal point method applied to the negative dual function \(\psi\), i.e.,
\begin{align} \tag{\ref{eq:ppa} rev.}
y^{k+1} &= \aprox[\lambda]{\phi}{\psi}(y^k) = \argmin_{y \in \bR^m}~\psi(y) + \lambda \star \phi(y^k - y),
\end{align}

the optimality conditions of each multiplier update yield, for $\phi = \frac{1}{p+1} \Vert \cdot \Vert_2^{p+1}$,
\begin{align} 
    0 \in \partial \psi(y^{k+1}) + \lambda^{-p} \Vert y^{k+1} - y^k \Vert^{p-1} (y^{k+1} - y^{k})\textcolor{red}{.}
\end{align}
For classical ALM with $p = 1$, this reduces to a classical proximal point method with optimality conditions $0 \in \partial \psi(y^{k+1}) + \lambda^{-1} (y^{k+1} - y^{k})$.
Therefore, power ALM with $2$-norms can be interpreted as classical ALM with an adaptive penalty scheme that is implicitly defined by 
\begin{equation} \label{eq:adaptive_penalty_2norm}
    \lambda_k := \lambda^{p} \Vert y^{k+1} - y^k \Vert^{1-p}.
\end{equation}
Notably, if \(\Vert y^{k+1} - y^k \Vert \to 0\) and \(p > 1\), then \(\lambda_k \to \infty\).
This is qualitatively similar to adaptive penalty schemes that are often used in the context of ALMs, see e.g.\,\cite[\S 17.4]{nocedal_numerical_2006} and \cite[\S 5.3]{hermans_qpalm_2021}.

In the case of $p$-norms, a similar interpretation holds, but the scalar penalty $\lambda_k$ becomes an adaptive diagonal penalty matrix $\Lambda_k$ implicitly defined by
\begin{equation} \label{eq:adaptive_penalty_pnorm}
    \left[ \Lambda_k \right]_{j, j} = \lambda^{p} \lvert y_j^{k+1} - y_j^k \rvert^{1-p}, \quad \forall j \in \{1, 2, \ldots, n\}.
\end{equation}

\section{Numerical evaluation}
\label{section:numerics}
This section compares the performance of power ALM, \cref{alg:inexact_al_averaging} with $\theta_k = 1$, on various composite optimization problems for powers $q+1$, with $q \in\{ 0.7, 0.8, 0.9, 1\}$.
Note that the choice $q = 1$ is equivalent to the classical ALM.
We also compare against classical ALM using an adaptive penalty parameter $\lambda_k$, since power ALM implicitly defines an adaptive penalty scheme (cf.\,\cref{sec:implicit-adaptive-penalty-interpretation}), and since such variants are known to yield good performance in practice, see e.g.\,\cite[\S $17.4$]{nocedal_numerical_2006} and \cite[\S $5.3$]{hermans_qpalm_2021}.
In particular, we consider an adaptive rule that sets $\lambda_{k+1} = 2 \lambda_k$ whenever the absolute constraint violation $r_{k+1}$ has not sufficiently decreased, i.e., ${r_{k+1}} \geq \delta {r_k}$ for some $\delta \in (0, 1)$, and sets $\lambda_{k+1} = \lambda_k$ otherwise.
This can be seen as a simplified version of the scheme proposed in \cite[\S $5.3$]{hermans_qpalm_2021}.
Unless mentioned otherwise, the choice $\delta = 0.1$ is used.
The solvers are terminated if $\vert f(x^{k+1}) - f^\star \vert \leq 10^{-6}$ and $r_{k+1} \leq 10^{-6}$.
As a termination criterion for the inner problems we use $\Vert \nabla_{x} L_\lambda (x^{k+1}, y^k) \Vert \leq \tfrac{10^{-3}}{k^{p+1}}$, i.e., the criterion supported by the theory.
The code to reproduce these experiments is available online\footnote{\url{https://github.com/alexanderbodard/Power-PPM-and-ALM}}.

\subsection{Linear programming}

\begin{table}
    \caption{
        Median of the total number of BFGS iterations required to solve the LP, computed over $N=20$ random realizations.
    }
    \label{table:lps}
    \centering
    \begin{adjustbox}{width=\textwidth}
    \setlength\extrarowheight{5pt}
    
    \pgfplotstabletypeset[%
        begin table={\begin{tabular}[t]},
        every head row/.style={
            before row={%
              \hline
              \vphantom{$q = 1$}\\
              \hline
              \vphantom{$q = 1$}\\
              \hline
            },
            after row/.add={}{\hline},
        },
        header=true,
        col sep=&,
        row sep=\\,
        string type,
        columns/{$(m, n)$}/.style ={column name={$(m, n)$}, column type={|c}},
        every row no 7/.style={after row=\hline},
    ]{
        \\
        $(m, n)$\\
        (200, 100)\\
        (400, 200)\\
        (600, 300)\\
        (300, 100)\\
        (600, 200)\\
        (400, 100)\\
        (500, 100)\\
        (600, 100)\\
    }%
    \begin{filecontents*}{data/lp.csv}
1.267000000000000000e+03 2.037500000000000000e+03 1.073500000000000000e+03 1.227000000000000000e+03 9.655000000000000000e+02 1.176000000000000000e+03 1.045000000000000000e+03 1.426500000000000000e+03
2.866000000000000000e+03 3.646000000000000000e+03 2.100000000000000000e+03 2.881000000000000000e+03 1.844000000000000000e+03 2.624500000000000000e+03 2.060500000000000000e+03 3.103500000000000000e+03
3.236000000000000000e+03 3.796500000000000000e+03 2.827000000000000000e+03 3.102500000000000000e+03 3.041500000000000000e+03 3.787500000000000000e+03 2.818000000000000000e+03 3.571500000000000000e+03
9.020000000000000000e+02 1.592000000000000000e+03 9.600000000000000000e+02 9.020000000000000000e+02 8.625000000000000000e+02 8.280000000000000000e+02 9.375000000000000000e+02 1.025000000000000000e+03
1.552000000000000000e+03 2.106000000000000000e+03 1.918500000000000000e+03 1.551500000000000000e+03 1.872500000000000000e+03 1.321000000000000000e+03 1.599500000000000000e+03 1.606500000000000000e+03
8.785000000000000000e+02 1.471000000000000000e+03 1.027500000000000000e+03 8.835000000000000000e+02 9.655000000000000000e+02 8.160000000000000000e+02 9.360000000000000000e+02 9.245000000000000000e+02
9.945000000000000000e+02 1.338500000000000000e+03 1.178500000000000000e+03 9.815000000000000000e+02 1.176500000000000000e+03 7.750000000000000000e+02 9.505000000000000000e+02 9.370000000000000000e+02
1.112500000000000000e+03 1.373500000000000000e+03 1.223000000000000000e+03 1.031000000000000000e+03 1.433000000000000000e+03 8.295000000000000000e+02 1.091500000000000000e+03 8.905000000000000000e+02
\end{filecontents*}
\pgfplotstabletypeset[%
        begin table={\begin{tabular}[t]},
        every head row/.style={
        before row={%
          \hline
          \multicolumn{4}{|c|}{Classical ALM} & \multicolumn{4}{|c|}{Power ALM}\\
          \hline
          \multicolumn{2}{|c|}{$q = 1$, fixed $\lambda$} & \multicolumn{2}{c|}{$q = 1$, adaptive $\lambda$} & \multicolumn{2}{c|}{$q \neq 1, \lambda = 10^2$} & \multicolumn{2}{c|}{$q \neq 1, \lambda = 10^3$}\\
          \hline
        },
        after row/.add={}{\hline},
        },
        header=true,
       precision=1,
       columns/0/.style ={column name={$\lambda = 10^3$}, column type={|l}},
       columns/1/.style ={column name={$\lambda = 10^4$}, column type={r|}},
       columns/2/.style ={column name={$\lambda_0 = 10^2$}, column type={l}},
       columns/3/.style ={column name={$\lambda_0 = 10^3$}, column type={r|}},
       columns/4/.style ={column name={$q = 0.9$}, column type={l}},
       columns/5/.style ={column name={$q = 0.8$}, column type={r|}},
       columns/6/.style ={column name={$q = 0.9$}, column type={l}},
       columns/7/.style ={column name={$q = 0.8$}, column type={r|}},
        every row no 7/.style={after row=\hline},
       every row 0 column 4/.style={
            postproc cell content/.style={
              @cell content/.add={$\bf}{$}
            }
        },
       every row 1 column 4/.style={
            postproc cell content/.style={
              @cell content/.add={$\bf}{$}
            }
        },
       every row 2 column 6/.style={
            postproc cell content/.style={
              @cell content/.add={$\bf}{$}
            }
        },
       every row 3 column 5/.style={
            postproc cell content/.style={
              @cell content/.add={$\bf}{$}
            }
        },
       every row 4 column 5/.style={
            postproc cell content/.style={
              @cell content/.add={$\bf}{$}
            }
        },
       every row 5 column 5/.style={
            postproc cell content/.style={
              @cell content/.add={$\bf}{$}
            }
        },
       every row 6 column 5/.style={
            postproc cell content/.style={
              @cell content/.add={$\bf}{$}
            }
        },
       every row 7 column 5/.style={
            postproc cell content/.style={
              @cell content/.add={$\bf}{$}
            }
        },
    ]{data/lp.csv}
    \end{adjustbox}
    
\end{table}

Consider a linear program with inequality constraints, which is an instance of \cref{eq:problem} via the choices $f(x):=\langle c, x\rangle$, $g:=\delta_{-\bR^m_+}$ and $\mathcal{A}(x)=Ax-b$
where $A \in \bR^{m \times n}$ and $b \in \bR^{m}$.
The problems are randomly generated such that $\text{cond}(A) = 1000$ using the method described in \cite{mohammadisiahroudi_generating_2023}. We used the implementation of the same authors which is available online\footnote{\url{https://github.com/QCOL-LU/QIPM}.}.
The inner problems are solved using BFGS.
\cref{table:lps} compares power ALM with $(q+1)$-norms in terms of the total number of BFGS iterations.
Classical ALM with adaptive penalty uses the proposed adaptive scheme with parameter $\delta = 10^{-3}$.
The table reports the median of the total number of BFGS iterations over $N = 20$ random realizations of \cref{eq:problem}, and this for various dimensions $(m, n)$ of the problem.
Additionally, \cref{table:lps-p95} in \cref{sec-append:additional-tables} provides the corresponding $95^\text{th}$ percentiles.

\subsection{Quadratic programming}

\begin{table}
    \caption{
        Median of the total number of iterations of \cite{li2024simpleuniformlyoptimalmethod} required to solve the QP, over $N=20$ random realizations.
    }
    \label{table:quadratics}
    \centering
    \begin{adjustbox}{width=\textwidth}
    \setlength\extrarowheight{5pt}
    
    \pgfplotstabletypeset[%
        begin table={\begin{tabular}[t]},
        every head row/.style={
            before row={%
              \hline
              \vphantom{$q = 1$}\\
              \hline
              \vphantom{$q = 1$}\\
              \hline
            },
            after row/.add={}{\hline},
        },
        header=true,
        col sep=&,
        row sep=\\,
        string type,
        columns/{$(m, n)$}/.style ={column name={$(m, n)$}, column type={|c}},
        every row no 9/.style={after row=\hline},
    ]{
        \\
        $(m, n)$\\
        (200, 400)\\
        (250, 500)\\
        (300, 600)\\
        (350, 700)\\
        (400, 800)\\
        (450, 900)\\
        (150, 450)\\
        (200, 600)\\
        (250, 750)\\
        (300, 900)\\
    }%
    \begin{filecontents*}{data/convex_quadratic_v2.csv}
4.264500000000000000e+03 4.157500000000000000e+03 8.434500000000000000e+03 5.689000000000000000e+03 3.980000000000000000e+03 3.639000000000000000e+03 3.923000000000000000e+03 4.123000000000000000e+03 4.229500000000000000e+03
4.127000000000000000e+03 4.793000000000000000e+03 1.027000000000000000e+04 5.589500000000000000e+03 4.177000000000000000e+03 4.071500000000000000e+03 3.962500000000000000e+03 3.838000000000000000e+03 4.137500000000000000e+03
4.189000000000000000e+03 4.608000000000000000e+03 1.236650000000000000e+04 5.692500000000000000e+03 4.473500000000000000e+03 4.907000000000000000e+03 3.800000000000000000e+03 3.671000000000000000e+03 4.250500000000000000e+03
4.493000000000000000e+03 5.184500000000000000e+03 1.314450000000000000e+04 5.671500000000000000e+03 4.672000000000000000e+03 5.160000000000000000e+03 4.379500000000000000e+03 3.946500000000000000e+03 4.476500000000000000e+03
4.198500000000000000e+03 4.824500000000000000e+03 1.472250000000000000e+04 5.248500000000000000e+03 4.663000000000000000e+03 5.420000000000000000e+03 3.781000000000000000e+03 3.843000000000000000e+03 4.290000000000000000e+03
4.646500000000000000e+03 4.887500000000000000e+03 1.579550000000000000e+04 5.749000000000000000e+03 4.141500000000000000e+03 5.621000000000000000e+03 4.200000000000000000e+03 3.915000000000000000e+03 4.334000000000000000e+03
5.002000000000000000e+03 5.263000000000000000e+03 1.112400000000000000e+04 6.165000000000000000e+03 5.058500000000000000e+03 4.626500000000000000e+03 4.222000000000000000e+03 4.365000000000000000e+03 4.378000000000000000e+03
5.178500000000000000e+03 5.067500000000000000e+03 1.415250000000000000e+04 6.695500000000000000e+03 4.748500000000000000e+03 5.758000000000000000e+03 4.615000000000000000e+03 4.905500000000000000e+03 4.376000000000000000e+03
5.957500000000000000e+03 5.647000000000000000e+03 1.622550000000000000e+04 6.785500000000000000e+03 4.902000000000000000e+03 6.024500000000000000e+03 5.310500000000000000e+03 6.994000000000000000e+03 4.014500000000000000e+03
6.073000000000000000e+03 6.020000000000000000e+03 1.922400000000000000e+04 6.596000000000000000e+03 5.017000000000000000e+03 6.615000000000000000e+03 6.094000000000000000e+03 8.778000000000000000e+03 4.140000000000000000e+03
\end{filecontents*}
\pgfplotstabletypeset[%
        begin table={\begin{tabular}[t]},
        every head row/.style={
        before row={%
          \hline
          \multicolumn{6}{|c|}{Classical ALM} & \multicolumn{3}{|c|}{Power ALM}\\
          \hline
          \multicolumn{3}{|c|}{$q = 1$, fixed $\lambda$} & \multicolumn{3}{c|}{$q = 1$, adaptive $\lambda$} & \multicolumn{3}{c|}{$q \neq 1, \lambda = 0.1$}\\
          \hline
        },
        after row/.add={}{\hline},
        },
        header=true,
       precision=1,
       columns/0/.style ={column name={$\lambda = 0.1$}, column type={|l}},
       columns/1/.style ={column name={$\lambda = 1$}, column type={c}},
       columns/2/.style ={column name={$\lambda = 10$}, column type={r|}},
       columns/3/.style ={column name={$\lambda_0 = 0.01$}, column type={l}},
       columns/4/.style ={column name={$\lambda_0 = 0.1$}, column type={c}},
       columns/5/.style ={column name={$\lambda_0 = 1$}, column type={r|}},
       columns/6/.style ={column name={$q = 0.9$}, column type={l}},
       columns/7/.style ={column name={$q = 0.8$}, column type={c}},
       columns/8/.style ={column name={$q = 0.7$}, column type={r|}},
        every row no 9/.style={after row=\hline},
       every row 0 column 5/.style={
            postproc cell content/.style={
              @cell content/.add={$\bf}{$}
            }
        },
       every row 1 column 7/.style={
            postproc cell content/.style={
              @cell content/.add={$\bf}{$}
            }
        },
       every row 2 column 7/.style={
            postproc cell content/.style={
              @cell content/.add={$\bf}{$}
            }
        },
       every row 3 column 7/.style={
            postproc cell content/.style={
              @cell content/.add={$\bf}{$}
            }
        },
       every row 4 column 6/.style={
            postproc cell content/.style={
              @cell content/.add={$\bf}{$}
            }
        },
       every row 5 column 7/.style={
            postproc cell content/.style={
              @cell content/.add={$\bf}{$}
            }
        },
       every row 6 column 6/.style={
            postproc cell content/.style={
              @cell content/.add={$\bf}{$}
            }
        },
       every row 7 column 8/.style={
            postproc cell content/.style={
              @cell content/.add={$\bf}{$}
            }
        },
       every row 8 column 8/.style={
            postproc cell content/.style={
              @cell content/.add={$\bf}{$}
            }
        },
       every row 9 column 8/.style={
            postproc cell content/.style={
              @cell content/.add={$\bf}{$}
            }
        },
       every row 10 column 7/.style={
            postproc cell content/.style={
              @cell content/.add={$\bf}{$}
            }
        },
       every row 11 column 7/.style={
            postproc cell content/.style={
              @cell content/.add={$\bf}{$}
            }
        },
       every row 12 column 7/.style={
            postproc cell content/.style={
              @cell content/.add={$\bf}{$}
            }
        },
       every row 13 column 7/.style={
            postproc cell content/.style={
              @cell content/.add={$\bf}{$}
            }
        },
       every row 14 column 7/.style={
            postproc cell content/.style={
              @cell content/.add={$\bf}{$}
            }
        },
       every row 15 column 7/.style={
            postproc cell content/.style={
              @cell content/.add={$\bf}{$}
            }
        },
       every row 16 column 7/.style={
            postproc cell content/.style={
              @cell content/.add={$\bf}{$}
            }
        },
    ]{data/convex_quadratic_v2.csv}
    \end{adjustbox}
    
\end{table}

We now consider random positive semi-definite quadratic programs. This is an instance of \cref{eq:problem} via the choices $f(x)=\delta_{[\ell, u]^n}(x) + \tfrac{1}{2} \langle x, Qx \rangle + \langle c, x\rangle$, $g=\delta_{\{0\}}$ and $\mathcal{A}(x)= Ax-b$.
The problems are generated as follows.
The vector $c \in \bR^{n}$ and a diagonal matrix $D \in \bR^{n \times n}$ are sampled from a normal distribution with unit variance, centered around zero and five respectively. 
Any negative elements of $D$ are clipped to $0$ to ensure positive semi-definiteness.
Then, a random integer $k$ is taken uniformly random from the interval $[\lceil \tfrac{n}{4} \rceil, \lfloor \tfrac{n}{2} \rfloor]$.
Next, $k$ diagonal elements of the matrix $D$, which are chosen uniformly at random without replacement, are set to zero.
We generate an orthonormal matrix $V$ as the orthonormal matrix in the QR-decomposition of a random matrix, which is taken from a standard normal distribution.
We define the positive semi-definite matrix $Q = V D V^\top$.
The constraint matrix $A \in \bR^{m \times n}$ is sampled from the standard normal distribution, and the vector $b \in \bR^{m}$ is sampled uniformly at random over the interval $[-1, 1]$.
The upper and lower bounds are fixed as $u = -\ell = 0.8 \cdot 1_n$.

The equality constraint $A x = b$ is enforced through the augmented Lagrangian, whereas the box-constraints are enforced in the inner problems using the first-order algorithm from \cite{li2024simpleuniformlyoptimalmethod}.
\cref{table:quadratics} compares power ALM with $2$-norms raised to a power $q+1$ in terms of the total number of iterations of \cite{li2024simpleuniformlyoptimalmethod}.
A reference solution $f^\star$ is computed by QPALM \cite{hermans_qpalm_2021} up to an absolute accuracy of (at least) $10^{-10}$.
The table reports both the median of the total number of iterations \cite{li2024simpleuniformlyoptimalmethod} over $N = 20$ random realizations of the QP, and this for various problem sizes $(m, n)$.
Observe that power ALM with a well-chosen power \(q < 1\) consistently outperforms classical ALM with fixed $\lambda$, and performs very similar to classical ALM with adaptive $\lambda$, if not better.


Additionally, we consider positive semi-definite quadratic programs in which $f(x) = \frac{1}{2} \langle x, Q x \rangle + \langle c, x \rangle$ and $g = \delta_{\bR_-}$.
Remark that we now have linear inequality constraints instead of linear equality constraints with box constraints.
The problems are generated as before, and we again use $2$-norms raised to the power $q+1$.
However, this time we use an L-BFGS \cite{byrd1995limited,zhu_algorithm_1997} inner solver.
\cref{fig:representative-plot} visualizes the suboptimality $\lvert f(x^{k+1}) - f^\star \rvert$ and the constraint violation $\Vert A x^{k+1} - b \Vert_2$ as a function of the performed number of L-BFGS iterations for a random instance of the QP with $(m, n) = (400, 800)$.
Moreover, it depicts the interpretation of power ALM as an implicitly defined adaptive penalty scheme by plotting the adaptive penalty $\lambda_k$ from \eqref{eq:adaptive_penalty_2norm}. In our experience, the fast initial decrease in both suboptimality and constraint violation is representative for power ALM with $q \neq 1$.

\begin{figure}
    \centering
    \begin{subfigure}[b]{0.32\textwidth}
        \centering
        \includegraphics[width=\textwidth]{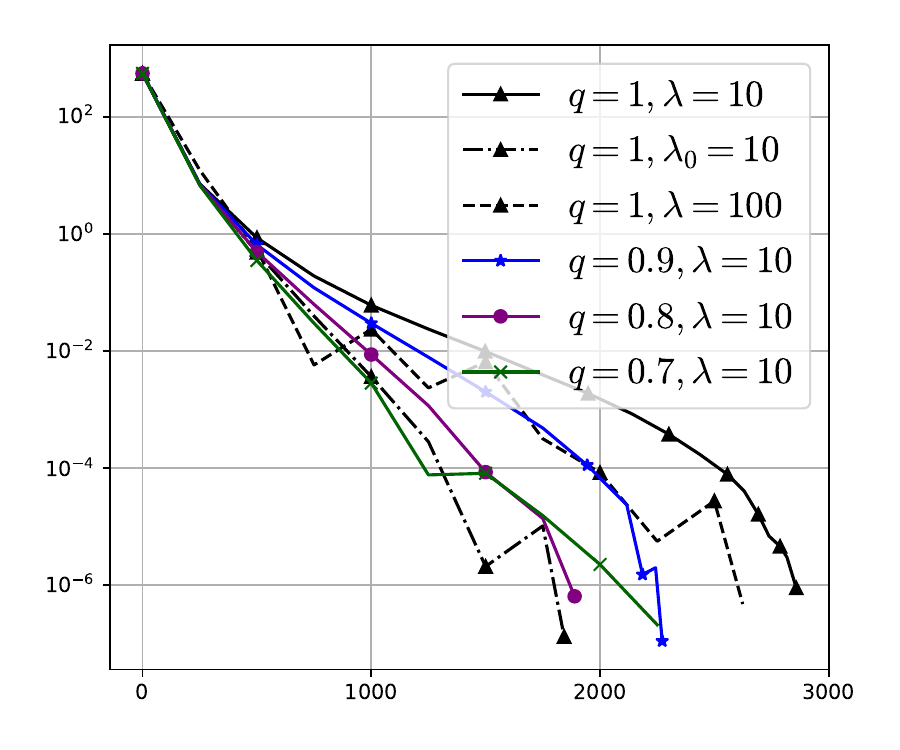}
        \captionsetup{justification=centering}
        \caption{Suboptimality\label{fig:representative-subopt}}
    \end{subfigure}
    \begin{subfigure}[b]{0.32\textwidth}
        \centering
        \includegraphics[width=\textwidth]{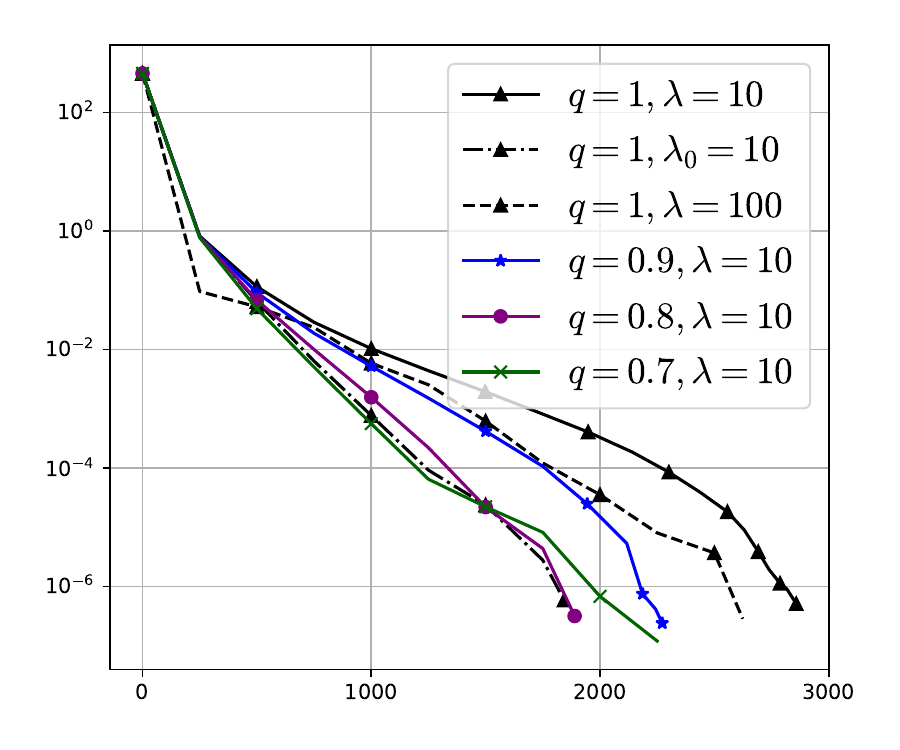}
        \captionsetup{justification=centering}
        \caption{Constraint violation\label{fig:representative-const}}
    \end{subfigure}
    \begin{subfigure}[b]{0.32\textwidth}
        \centering
        \includegraphics[width=\textwidth]{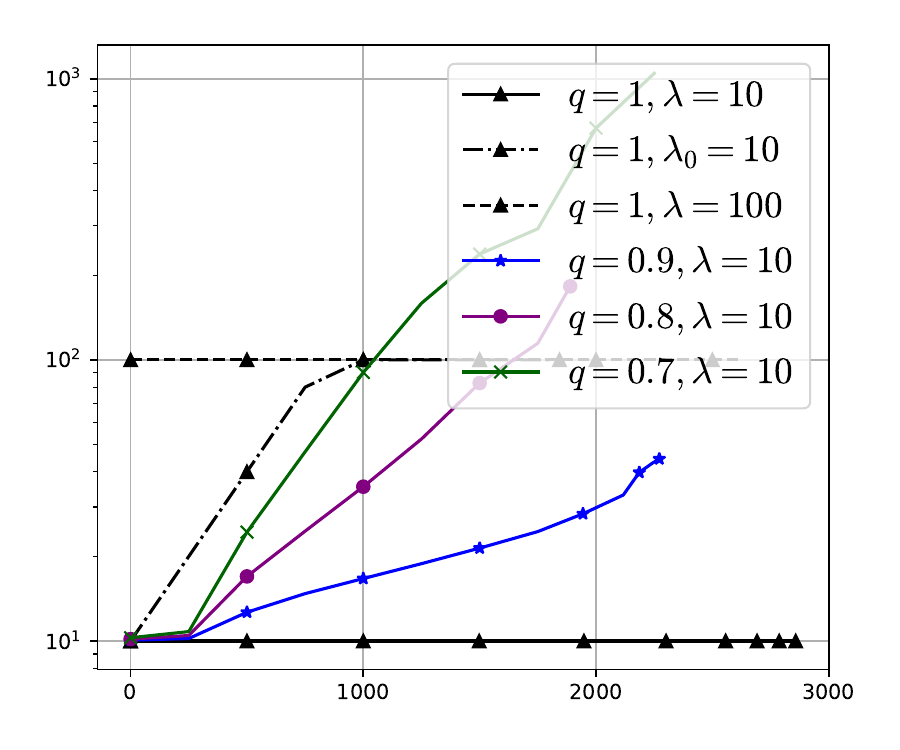}
        \captionsetup{justification=centering}
        \caption{Adaptive penalty\label{fig:representative-penalty}}
    \end{subfigure}
    \captionsetup{justification=centering}
    \caption{Power ALM with L-BFGS inner solver on a QP instance ($m = 800$, $n = 400$). The horizontal axis corresponds to the cumulative number of inner iterations using L-BFGS.\label{fig:representative-plot}}
\end{figure}

\subsection{Regularized \texorpdfstring{$\ell_1$}{l_1}-regression problem}

Finally, we consider an $\ell_2$-regularized $\ell_1$-regression problem, which is an instance of \cref{eq:problem} with $f(x) = \frac{\theta}{2} \Vert x \Vert_2^2$, $g = \Vert \cdot \Vert_1$, and $\mathcal A(x) = A x - b$.
The entries of the matrix $A \in \bR^{m \times n}$ are sampled uniformly at random over the interval $[-5, 5]$, the vector $b \in \bR^{m}$ is sampled from the standard normal distribution, and we fix $\theta = 100$. 
As an inner solver, we use the method presented in \cite{li2024simpleuniformlyoptimalmethod}.
Classical ALM with adaptive penalty uses the proposed adaptive scheme with parameter $\delta = 10^{-1}$.
\Cref{fig:l1-regression-pnorm} visualizes the convergence power ALM when using $(q+1)$-norms, both in terms of the primal cost and the primal-dual gap. It illustrates that qualitatively, these unconventional powers perform similar to a classical ALM method with an adaptive penalty scheme.
\Cref{fig:l1-regression-2norm} depicts similar convergence results for power ALM when using $2$-norms, and visualizes the adaptive penalty interpretation discussed in \cref{sec:implicit-adaptive-penalty-interpretation}. 

\begin{figure}[h]
    \centering
    \begin{subfigure}[b]{0.45\textwidth}
        \centering
        \includegraphics[width=\textwidth]{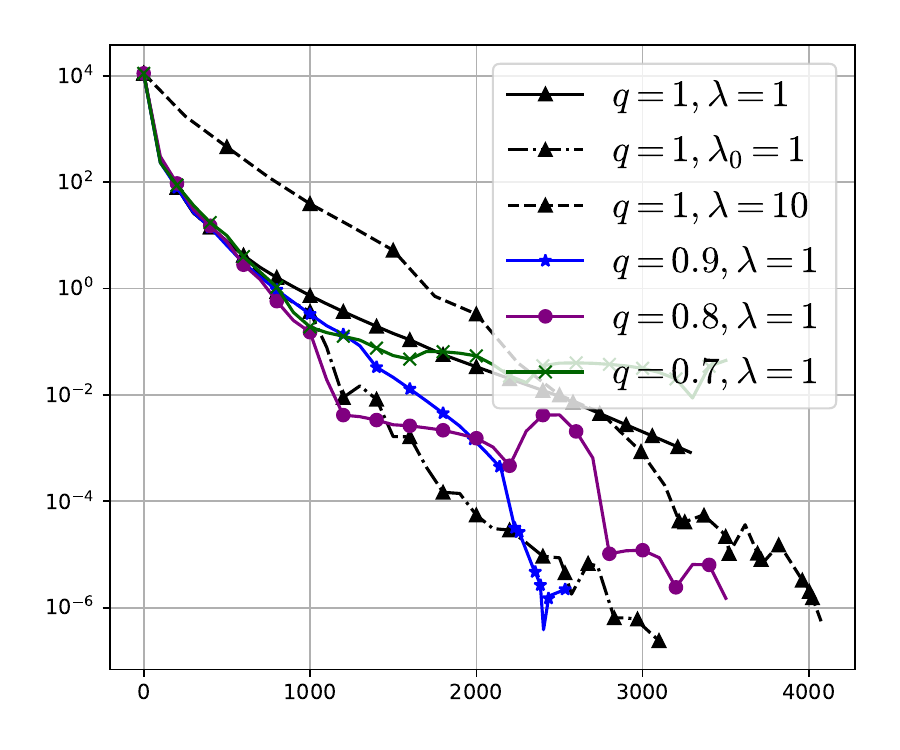}
        \captionsetup{justification=centering}
        \caption{Primal cost\label{fig:l1-regression-pnorm-primal-cost}}
    \end{subfigure}
    \begin{subfigure}[b]{0.45\textwidth}
        \centering
        \includegraphics[width=\textwidth]{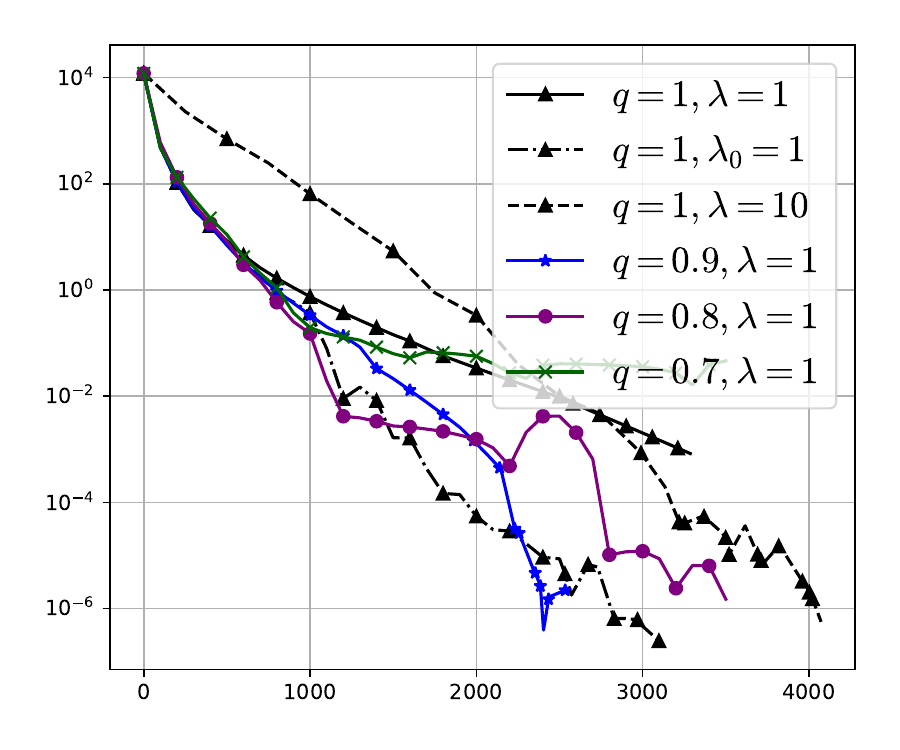}
        \captionsetup{justification=centering}
        \caption{Primal-dual gap\label{fig:l1-regression-pnorm-primal-dual-gap}}
    \end{subfigure}
    \captionsetup{justification=centering}
    \caption{Regularized $\ell_1$-regression problem with $(q+1)$-norms. The horizontal axis corresponds to the cumulative number of inner iterations using \cite{li2024simpleuniformlyoptimalmethod}. \label{fig:l1-regression-pnorm}}
\end{figure}

\begin{figure}
    \centering
    \begin{subfigure}[b]{0.32\textwidth}
        \centering
        \includegraphics[width=\textwidth]{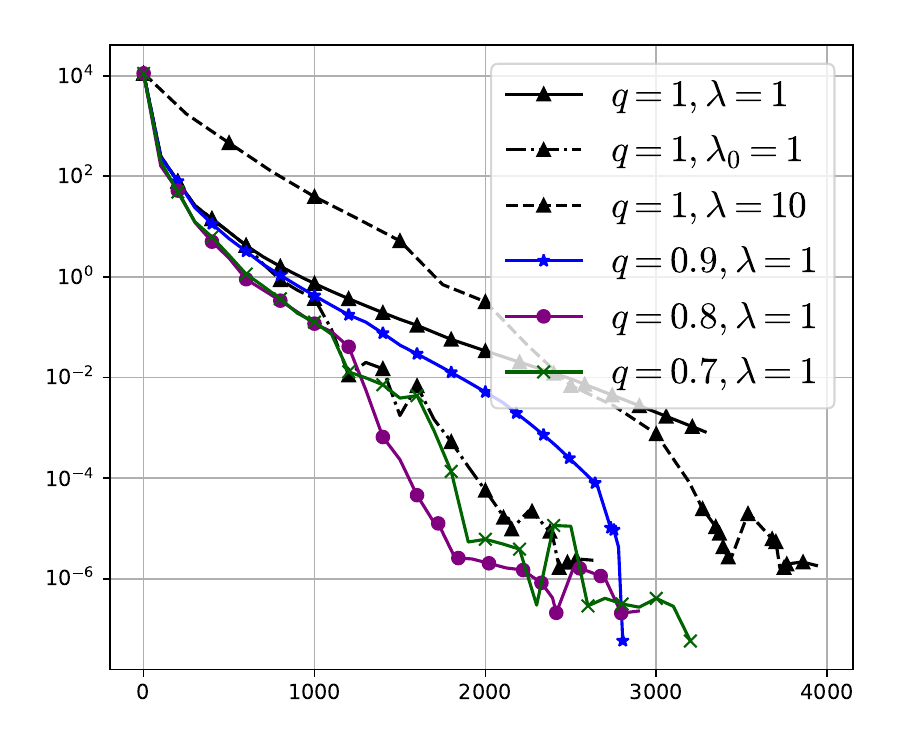}
        \captionsetup{justification=centering}
        \caption{Primal cost\label{fig:l1-regression-2norm-primal-cost}}
    \end{subfigure}
    \begin{subfigure}[b]{0.32\textwidth}
        \centering
        \includegraphics[width=\textwidth]{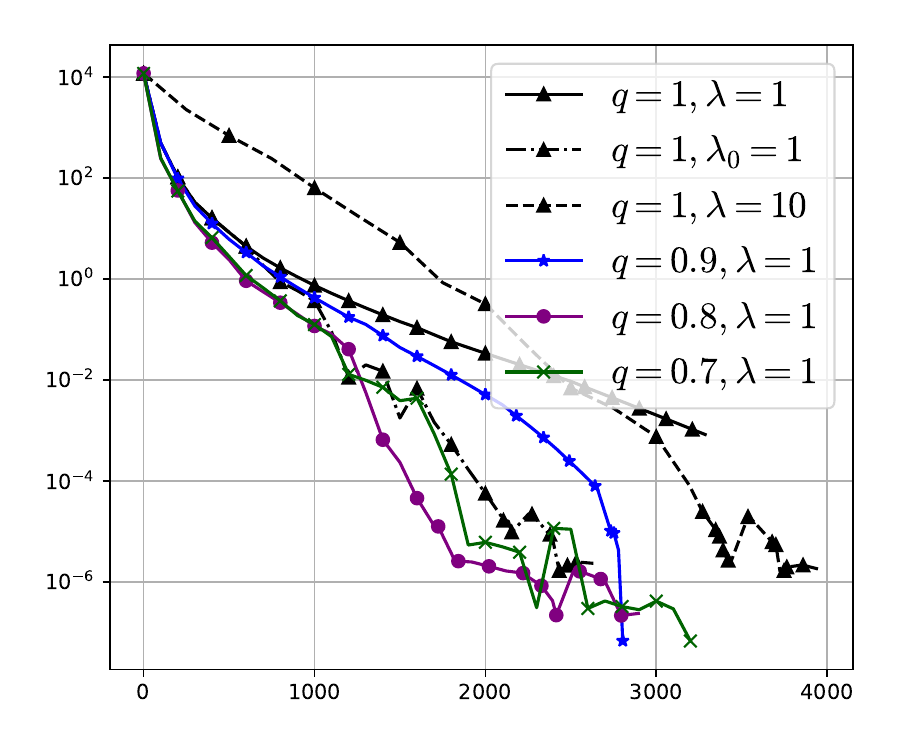}
        \captionsetup{justification=centering}
        \caption{Primal-dual gap\label{fig:l1-regression-2norm-primal-dual-gap}}
    \end{subfigure}
    \begin{subfigure}[b]{0.32\textwidth}
        \centering
        \includegraphics[width=\textwidth]{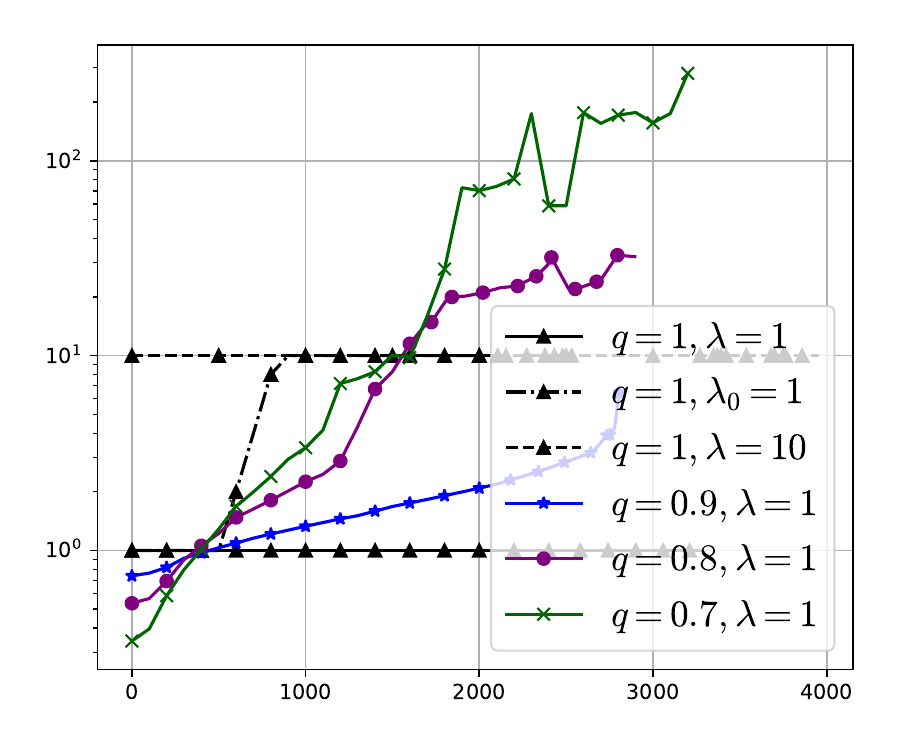}
        \captionsetup{justification=centering}
        \caption{Adaptive penalty\label{fig:l1-regression-2norm-penalty}}
    \end{subfigure}
    \captionsetup{justification=centering}
    \caption{Regularized $\ell_1$-regression problem with 2-norms. The horizontal axis corresponds to the cumulative number of inner iterations using \cite{li2024simpleuniformlyoptimalmethod}. \label{fig:l1-regression-2norm}}
\end{figure}

\section{Conclusion} \label{section:conclusion}
In this paper we have analyzed the inexact power proximal point method and the corresponding Augmented Lagrangian scheme under two different absolute inexactness conditions. In order to incorporate a more relaxed error tolerance, we have also studied the algorithm as an inexact anisotropic gradient descent method. Under this framework, we have generalized the results of \cite{bertsekas2014constrained} for the proximal point method which were specialized to the Augmented Lagrangian setting. Moreover, in the Augmented Lagrangian setting we showed ways to enforce the proposed error conditions in practice, while also providing the corresponding global rates for the primal-dual gap and the measure of infeasibility.
We also provide an interpretation of the power Augmented Lagrangian method as a classical ALM equipped with an implicit adaptive penalty rule.

Finally, our experiments illustrate that for certain powers the power ALM performs similarly to a fine-tuned classical ALM with adaptive penalty scheme, despite requiring fewer parameters.


%% file: proofs/lemma_decrease.tex
\begin{proof}
Let $y \in \bR^m$. Since $v \in \partial_{\varepsilon}\psi(y^+)$ we have by definition of the $\varepsilon$-subdifferential:
\begin{align*}
    \psi(y) \geq \psi(y^+) + \langle v, y - y^+\rangle - \varepsilon.
\end{align*}
Rearranging yields $\psi(y^+) \leq \psi(y) - \langle v, y - y^+\rangle + \varepsilon.$
Since $v = \nabla \phi(\lambda^{-1}(w - y^+)) = \lambda^{-p} \nabla \phi(w-y^+)$ we obtain
\begin{align*}
\psi(y^+) \leq \psi(y) + \lambda^{-p}\langle \nabla \phi(w-y^+), (w - y) -( w- y^+)\rangle + \varepsilon.
\end{align*}
In light of \cref{thm:uniform_convexity} we can further bound
\begin{align} \label{eq:uniform_convexity_ineq_full}
\psi(y^+) \leq \psi(y) + \lambda^{-p}\phi(w-y) - \lambda^{-p}\phi(w-y^+) - \tfrac{\lambda^{-p}}{2^{p-1}} \phi(y^+ -y) + \varepsilon.
\end{align}
Neglecting $- \lambda^{-p}\phi(w-y^+)$ we obtain the claimed result.
\ifx\ifsvjour\true
\qed
\fi
\end{proof}

%% file: proofs/global_subgradient_averaging.tex
\begin{proof}
Regardless of the choice of $\theta_k$, by definition of \cref{alg:inexact_pppa_averaging} we have that $v^k \in \partial_{\varepsilon_k}\psi(y^{k+1})$ and $y^{k+1} = w^k - \lambda \nabla \phi^*(v^k)$.
Define $y := \tfrac{a_{k+1}}{A_{k+1}}y^\star +  \tfrac{A_{k}}{A_{k+1}}y^k$. Invoking \cref{thm:decrease_inexact_ppa} and using the convexity of $\psi$ we get:
\begin{align} \label{eq:main_inequality_pppa_averaging}
   \psi(y^{k+1}) &\leq \psi(\tfrac{a_{k+1}}{A_{k+1}}y^\star + \tfrac{A_{k}}{A_{k+1}} y^k) + \lambda \star \phi(w^k-y) + \varepsilon_{k} \notag \\
   &\leq \tfrac{a_{k+1}}{A_{k+1}}\psi(y^\star) + \tfrac{A_{k}}{A_{k+1}} \psi(y^k) + \lambda \star \phi(w^k-y) + \varepsilon_{k}.
\end{align}
``\labelcref{thm:inexact_ppa_rate:averaging}'':
For $\theta_k := \tfrac{A_k}{A_{k+1}}$, we have by definition of $y$ and $w^k$:
\begin{align*}
w^k-y &=  \theta_k y^k + (1 - \theta_k)y^0 -(1 - \theta_k)y^\star - \theta_k y^k =(1 - \theta_k)(y^0-y^\star).
\end{align*}
Substituting this identity in \cref{eq:main_inequality_pppa_averaging} we obtain, since $\lambda \star \phi$ is positively homogeneous with order $p+1$, i.e., $\lambda \star\phi(\alpha x) = |\alpha|^{p+1}(\lambda \star\phi)(x)$ and $1 -\theta_k=\tfrac{a_{k+1}}{A_{k+1}}$:
\begin{align*}
   \psi(y^{k+1}) &\leq \tfrac{a_{k+1}}{A_{k+1}}\psi(y^\star) + \tfrac{A_{k}}{A_{k+1}} \psi(y^k) + (\tfrac{a_{k+1}}{A_{k+1}})^{p+1}(\lambda \star \phi)(y^0-y^\star) + \varepsilon_{k}.
\end{align*}
Rearranging and multiplying both sides with $A_{k+1} > 0$ yields
\begin{align*}
   A_{k+1}\big(\psi(y^{k+1})-\psi(y^\star)\big) &\leq A_{k} \big(\psi(y^k)-\psi(y^\star)\big)
   \\ &\qquad+ A_{k+1} (\tfrac{a_{k+1}}{A_{k+1}})^{p+1}(\lambda \star \phi)(y^0 - y^\star) + A_{k+1}\varepsilon_{k}.
\end{align*}
Summing the inequality from $k=0$ to $k=K-1$ we obtain by telescoping, since $A_0=0$, $A_k = k^{p+1}$ and $\varepsilon_{k} = \tfrac{c}{A_{k+1}}$:
$
   A_{K}\big(\psi(y^{K})-\psi(y^\star)\big) \leq \lambda \star \phi(y^0 - y^\star)\sum_{k=1}^{K} \tfrac{a_{k}^{p+1}}{A_{k}^p}+ Kc.
$
In light of \cite[Equation 35]{doikov2020inexact} we have $\sum_{k=1}^{K} \tfrac{a_{k}^{p+1}}{A_{k}^p} \leq (p+1)^{p+1}K$
and hence
\begin{align*}
   A_{K}\big(\psi(y^{K})-\psi(y^\star)\big) &\leq (p+1)^{p+1}K (\lambda \star \phi)(y^0 - y^\star) + Kc.
\end{align*}
Dividing the inequality by $A_{K}$ yields:
\[
   \psi(y^{K})-\psi(y^\star) \leq \tfrac{(p+1)^{p+1}(\lambda \star \phi)(y^0 - y^\star)}{K^{p}} + \tfrac{c}{K^{p}} =\tfrac{(p+1)^{p+1} (\lambda \star\phi)(y^0 - y^\star) + c}{K^{p}}.      
\]

``\labelcref{thm:inexact_ppa_rate:simple}'': 
For $\theta_k := 1$, we have that $w^k = y^k$ and since $y^k-y=\tfrac{a_{k+1}}{A_{k+1}}(y^k-y^\star)$ using the positive homogeneity of $\phi$ with order $p+1$, we obtain that
\begin{align*}
    \psi(y^{k+1}) \leq \tfrac{a_{k+1}}{A_{k+1}}\psi(y^\star) + \tfrac{A_{k}}{A_{k+1}}\psi(y^k) + (\tfrac{a_{k+1}}{A_{k+1}})^{p+1}(\lambda \star \phi)(y^k-y^\star) + \varepsilon_k.
\end{align*}
Multiplying both sides with $A_{k+1}$ we have since $a_{k+1}= A_{k+1} - A_k$
\begin{align*}
    A_{k+1}\big(\psi(y^{k+1})-\psi(y^\star)\big) \leq A_{k}\big(\psi(y^k)-\psi(y^\star)\big) +  \tfrac{a_{k+1}^{p+1}}{A_{k+1}^p}(\lambda \star\phi)(y^k-y^\star) + \varepsilon_k A_{k+1}.
\end{align*}
Summing the inequality from $k=0$ to $k=K-1$ we obtain since $A_0 = 0$:
\begin{align} \label{eq:decrease_inexact_ppa_main}
    A_{K}\big(\psi(y^{K}) - \psi(y^\star)\big) \leq \sum_{k=0}^{K-1} \tfrac{a_{k+1}^{p+1}}{A_{k+1}^p}(\lambda \star \phi)(y^k-y^\star) + \varepsilon_k A_{k+1}.
\end{align}
From \cref{thm:decrease_inexact_ppa} for $w=y=y^k$ we have $\psi(y^{k+1}) \leq \psi(y^k)+ \varepsilon_k.$
Summing the inequality from $k=0$ to $k=K-1$ we obtain $
    \psi(y^{K})-\psi(y^0) \leq \sum_{k=0}^{K-1}\varepsilon_k \leq \beta,$
and hence $\|y^K-y^\star\| \leq D_\beta$ for any $K \geq 0$. This implies that $\phi(y^k-y^\star) =\tfrac{1}{p+1}\|y^\star - y^k\|^{p+1} \leq \tfrac{1}{p+1}D_\beta^{p+1}$.
Thus we can further bound \cref{eq:decrease_inexact_ppa_main}:
\begin{align*} 
    A_{K}\big(\psi(y^{K}) - \psi(y^\star)\big) \leq \lambda^{-p}\tfrac{1}{p+1}D_\beta^{p+1} \sum_{k=1}^{K} \tfrac{a_{k}^{p+1}}{A_{k}^p} + \varepsilon_k A_{k+1}.
\end{align*}
We choose $A_{K}=K^{p+1}$ and by using the fact that $\sum_{k=1}^{K} \tfrac{a_{k}^{p+1}}{A_{k}^p} \leq (p+1)^{p+1}K$ we get:
\begin{align*} 
    A_{K}\big(\psi(y^{K}) - \psi(y^\star)\big) \leq \lambda^{-p}D_\beta^{p+1} (p+1)^{p}K + c K.
\end{align*}
Dividing by $A_K$ we obtain $\psi(y^{K}) - \psi(y^\star) \leq \tfrac{\lambda^{-p}D_\beta^{p+1} (p+1)^{p} + c}{K^p}.$
\ifx\ifsvjour\true
\qed
\fi
\end{proof}

%% file: proofs/global_prox.tex
\begin{proof}
``\labelcref{lemma:env_convergence:ineq}'':
By \cite[Proposition 4.1]{laude2025anisotropic}, $\aenv{\lambda\star\phi}{\psi}(y)$ is anisotropically smooth relative to $\phi$ with constant $\tfrac{1}{\lambda}$. Hence the anisotropic descent inequality \cref{eq:adescent} at points $y = y^{k+1}$ and $\eta = y^k$ yields via the relation \cref{eq:gradient_formula}:
\begin{align} \label{eq:descent_ineq_env} \nonumber
\aenv{\lambda\star\phi}{\psi}(y^{k+1}) 
&\leq \nonumber \aenv{\lambda\star\phi}{\psi}(y^{k}) + \lambda \star \phi\big(y^{k+1} - y^k + \lambda \nabla \phi^*(\nabla \phi(\lambda^{-1}(y^k - \aprox{\lambda\star\phi}{\psi}(y^k))))\big)  
\\&\qquad -\lambda \star \phi\big(\lambda \nabla \phi^*(\nabla \phi(\lambda^{-1}(y^k - \aprox{\lambda\star\phi}{\psi}(y^k))))\big)
\nonumber
\\ &\leq \nonumber \aenv{\lambda\star\phi}{\psi}(y^{k}) + \lambda \star \phi(y^{k+1} - y^k + y^k - \aprox{\lambda\star\phi}{\psi}(y^k)) - \lambda \star \phi(y^k - \aprox{\lambda\star\phi}{\psi}(y^k)) \nonumber \\
&\leq \aenv{\lambda\star\phi}{\psi}(y^{k}) + \varepsilon_k -\lambda \star \phi(y^k - \aprox{\lambda\star\phi}{\psi}(y^k)), 
\end{align}
where the last inequality follows by the error bound \eqref{eq:env_error_cond}.

``\labelcref{lemma:env_convergence:diff}'': Since $\argmin \phi = \{0\}$, we have that $\aenv{\lambda\star\phi}{\psi}(y^\star) = \psi(y^\star)$ and as such that $\psi^\star \leq \aenv{\lambda\star\phi}{\psi}(y)$ for all $y \in \bR^m$.
Thus by summing \cref{eq:descent_ineq_env} form $k=0$ to $K$, we get:
\begin{align*}
-\infty &< \psi(y^\star) -\aenv{\lambda\star\phi}{\psi}(y^0) \leq \aenv{\lambda\star\phi}{\psi}(y^{K+1}) -\aenv{\lambda\star\phi}{\psi}(y^0) \\
&\leq- \sum_{k=0}^K \lambda \star \phi(y^k - \aprox{\lambda\star\phi}{\psi}(y^k)) + \sum_{k=0}^K \varepsilon_k \leq - \sum_{k=0}^K \lambda \star \phi(y^k - \aprox{\lambda\star\phi}{\psi}(y^k)) + \alpha,
\end{align*}
where the last inequality follows by the assumed bound on $\varepsilon_k$.
This result readily implies that $\lambda \star \phi(y^k - \aprox{\lambda\star\phi}{\psi}(y^k))$ is summable and therefore
$
\lambda \star \phi(y^k - \aprox{\lambda\star\phi}{\psi}(y^k)) \to 0.
$
Moreover, from \cref{eq:env_error_cond} we have that
$
0 \leq \lambda \star \phi(y^{k+1} - \aprox{\lambda\star\phi}{\psi}(y^k))  \leq \varepsilon_k \to 0.
$
Recall that $\lambda \star \phi = \lambda^{-p}\frac{1}{p+1}\|\cdot\|^{p+1}$.
Hence, we can combine these two convergence results to show the claimed relationship, using the triangle inequality:
$$
\|y^k - y^{k+1}\| \leq \|y^k - \aprox{\lambda\star\phi}{\psi}(y^k) \| + \|y^{k+1} - \aprox{\lambda\star\phi}{\psi}(y^k)\| \to 0.
$$

``\labelcref{lemma:env_convergence:convergence}'':
Regarding the boundedness of the iterates we follow the proof of \cref{thm:inexact_ppa_rate:simple}. We have that for any $k \in \bN_0$
$
    \aenv{\lambda\star\phi}{\psi}(y^{k+1}) -\aenv{\lambda\star\phi}{\psi}(y^0) \leq  \alpha.
$
We define the diameter of the enlarged initial level-set of the envelope function:
\begin{align} \label{eq:env_level_sets}
    D_\alpha := \sup \{ \|y^\star - y\| : \aenv{\lambda\star\phi}{\psi}(y) \leq \aenv{\lambda\star\phi}{\psi}(y^0) + \alpha\}.
\end{align}
Since $\aenv{\lambda\star\phi}{\psi}$ is convex and $Y^\star$ is bounded, $D_\alpha$ is also bounded. Therefore, $\|y^k - y^\star\| \leq D_\alpha$ and this implies that $\{y^k\}_{k \in \bN_0}$ remains bounded. The subsequential convergence to $Y^\star$ follows from the fact that $\nabla \aenv{\lambda\star\phi}{\psi}(y^k) \to 0$. 

Now, we can turn our attention to the convergence rate of the anisotropic envelope of $\psi$. Using the convexity inequality, we get:
\begin{align*}
    \aenv{\lambda\star\phi}{\psi}(y) \geq \aenv{\lambda\star\phi}{\psi}(y^k) + \langle \nabla \aenv{\lambda\star\phi}{\psi}(y^k),y-y^k \rangle
\end{align*}
Rearranging and substituting \cref{eq:gradient_formula}, we have that
\begin{align*}
    \aenv{\lambda\star\phi}{\psi}(y^k) &\leq \aenv{\lambda\star\phi}{\psi}(y) - \lambda^{-p} \langle \nabla \phi(y^k - \aprox{\lambda\star\phi}{\psi}(y^k)),y-y^k \rangle \\
    & \leq \aenv{\lambda\star\phi}{\psi}(y) + \lambda^{-p} (\tfrac{1}{\sigma}\|\nabla \phi(y^k - \aprox{\lambda\star\phi}{\psi}(y^k))\|_*) (\sigma\|y-y^k\|),
\end{align*}
for any $\sigma > 0$, where the last inequality follows by H\"older's inequality. We can further bound the r.h.s. of the previous inequality, by using Young's inequality with exponents $q+1$ and $p+1$, the fact that $\|\nabla \phi(y)\|_* = \|y\|^p$ and $p(q+1) = p+1$:
\begin{align*}
    \aenv{\lambda\star\phi}{\psi}(y^k) \leq \aenv{\lambda\star\phi}{\psi}(y) + (\tfrac{1}{\sigma})^{q+1}\tfrac{\lambda^{-p}}{q+1}\|y^k - \aprox{\lambda\star\phi}{\psi}(y^k)\|^{p+1} + \sigma^{p+1}\tfrac{\lambda^{-p}}{p+1}\|y-y^k\|^{p+1}.
\end{align*}
Let us choose $\sigma = (\tfrac{p+1}{q+1})^{\tfrac{1}{q+1}} = p^{\tfrac{1}{q+1}}$. Then we have since $(\tfrac{1}{\sigma})^{q+1}\tfrac{1}{q+1} = \frac{1}{p+1}$ and $\sigma^{p+1} = p^p$:
\begin{align*}
    \aenv{\lambda\star\phi}{\psi}(y^k) \leq \aenv{\lambda\star\phi}{\psi}(y) + \lambda \star \phi(y^k - \aprox{\lambda\star\phi}{\psi}(y^k)) + p^p\lambda \star \phi(y - y^k),
\end{align*}
and by putting this result back in \cref{eq:descent_ineq_env} we obtain: 
\begin{align} \label{eq:env_ineq_rate}
    \aenv{\lambda\star\phi}{\psi}(y^{k+1}) \leq \aenv{\lambda\star\phi}{\psi}(y) + \varepsilon_k + p^p \lambda \star \phi(y - y^k).
\end{align}
The rest of the proof is similar to that of \cref{thm:inexact_ppa_rate:simple}. More specifically, if we choose $y := \tfrac{a_{k+1} y^\star + A_k y^k}{A_{k+1}}$ in \cref{eq:env_ineq_rate} and use the convexity inequality as well as the fact that $\phi$ is  positively homogeneous with order $p+1$, we have that
\begin{align*}
    \aenv{\lambda\star\phi}{\psi}(y^{k+1}) \leq \tfrac{a_{k+1}}{A_{k+1}}\aenv{\lambda\star\phi}{\psi}(y^\star) + \tfrac{A_k}{A_{k+1}}\aenv{\lambda\star\phi}{\psi}(y^k) + \varepsilon_k + \tfrac{a_{k+1}^{p+1}}{A_{k+1}^{p+1}}p^p\lambda \star\phi(y^k-y^\star)
\end{align*}
Multiplying the inequality above with $A_{k+1}>0$, rearranging and summing from $k=0$ to $k=K-1$, we get:
\begin{align} \label{eq:decrease_inexact_ppa_env}
    A_{K}\big(\aenv{\lambda\star\phi}{\psi}(y^K) - \aenv{\lambda\star\phi}{\psi}(y^\star) \big) \leq \sum_{k=0}^{K-1} \tfrac{a_{k+1}^{p+1}}{A_{k+1}^p} p^p \lambda \star \phi(y^k-y^\star) + \varepsilon_k A_{k+1}.
\end{align}
If we choose $A_K = K^{p+1}$ and use the fact that $\sum_{k=1}^{K} \tfrac{a_{k}^{p+1}}{A_{k}^p} \leq (p+1)^{p+1}K$ we have that
$
    A_{K}\big(\aenv{\lambda\star\phi}{\psi}(y^{k}) - \aenv{\lambda\star\phi}{\psi}(y^\star) \big) \leq p^p D_\alpha^{p+1} (p+1)^{p}K + c K,
$
since $\|y^k-y^\star\|\leq D_\alpha$. Dividing now with $A_K$, we obtain the claimed result.
\ifx\ifsvjour\true
\qed
\fi
\end{proof} 

%% file: proofs/superlinear.tex
\begin{proof}
    For a vector $y \in \bR^m$, we denote by $y_\star$ its unique projection on $Y^\star$, i.e. $y_\star = \argmin_{\eta \in Y^\star}\|y-\eta\|$. First, let us remark that since we have assumed the exact execution of the algorithm, the results of \cref{lemma:env_convergence} hold. More specifically, since $Y^\star$ is assumed compact, every subsequence of $\{y^k\}_{k \in \bN_0}$ converges to some point in $Y^\star$ and thus $d(y^k, Y^\star) \to 0$. Therefore, there exists a $K \geq 0$ such that for all $k \geq K$, $y^{k+1} \in \mathcal{N}(Y^\star, \delta)$. Hence we have that 
    $\mu d
    ^\nu(y^{k+1}, Y^\star) \leq \psi(y^{k+1})-\psi^\star$. By \cref{eq:exact_alg_update_y_full} we have that $-\lambda^{-p}\nabla \phi(y^{k+1} - y^k) \in \partial \psi(y^{k+1})$ and hence using the convexity inequality we obtain that $\psi^\star \geq \psi(y^{k+1}) - \lambda^{-p}\langle\nabla \phi(y^{k+1} - y^k),y^{k+1}_\star - y^{k+1} \rangle$. Thus we have:
    \begin{align}
    \mu d^\nu(y^{k+1}, Y^\star) &\leq -\psi^\star + \psi(y^{k+1}) \nonumber\\
    &\leq
    \lambda^{-p} \langle \nabla \phi(y^{k+1} - y^k), y^{k+1}_\star -y^{k+1}\rangle \nonumber \\
    & \leq \lambda^{-p} \|\nabla \phi(y^{k+1} - y^k)\|_{*}\|y^{k+1}-y^{k+1}_\star\| \nonumber \\
    &= \lambda^{-p} \|y^{k+1} - y^k\|^p d(y^{k+1}, Y^\star), \label{eq:superlinear_bound_p_norm}
    \end{align}
    where the third inequality follows from H\"older's inequality and the equality follows from the fact that $\|\nabla \phi(x)\|_{*} = \|x\|^p$ for any $x \in \bR^m$. Now, for the update of the algorithm we have that $y^{k+1} = \argmin_{y \in \bR^m}\{\psi(y) + \lambda \star \phi(y-y^k)\}$ and hence:
\begin{align*}
    \lambda \star \phi(y^{k+1} - y^k) \leq \psi(y)-\psi(y^{k+1}) + \lambda \star \phi(y^k - y), \qquad \forall y \in \bR^m.
\end{align*}
Thus, by choosing $y = y^k_\star \in Y^\star$, since $\psi^\star - \psi(y^{k+1})\leq 0$ we obtain since $\lambda \star \phi = \lambda^{-p}\|\cdot\|^{p+1}$:
\begin{align*}
    \|y^{k+1} - y^k\| \leq \|y^k - y^k_\star\| = d(y^k, Y^\star).
\end{align*}
substituting the inequality above in \cref{eq:superlinear_bound_p_norm}, we get:
\begin{align*}
    d^{\nu-1}(y^{k+1}, Y^\star) \leq \tfrac{1}{\lambda^{p} \mu}d^p(y^k, Y^\star)
\end{align*}
and by taking both sides to the power of $\frac{1}{\nu-1}$ we obtain the desired result noting that $\tfrac{1}{\lambda^{p} \mu}^{\frac{1}{\nu -1}} = (\tfrac{1}{\lambda \mu^{q}})^{\omega}$ and $\frac{1}{p}=q$.
\ifx\ifsvjour\true
\qed
\fi
\end{proof}

%% file: proofs/primal_dual_gap.tex
\begin{proof}
Choose any $w \in \bR^m$. We have by the $y$-update \cref{eq:exact-power-alm-dual} that $v^k = \lambda^{-p}\nabla \phi(w^k - y^{k+1}) = \nabla \phi(\lambda^{-1}(w^k - y^{k+1})) \in \partial_y (-L)(x^{k+1}, y^{k+1})$ and hence via concavity of $L(x^{k+1}, \cdot)$:
\begin{align}
L(x^{k+1}, w) \leq L(x^{k+1}, y^{k+1}) + \lambda^{-p}\langle \nabla \phi(w^k - y^{k+1}), (w^k - w) -(w^k -y^{k+1}) \rangle.
\end{align}
By uniform convexity of $\phi$ invoking \cref{thm:uniform_convexity} we can further bound:
\begin{align*}
L(x^{k+1}, w) &\leq L(x^{k+1}, y^{k+1}) + \lambda^{-p}\phi(w^k - w) - \lambda^{-p}\phi(w^k -y^{k+1}) - \lambda^{-p}\tfrac{1}{2^p}\phi(y^{k+1} - w) \\
&\leq L(x^{k+1}, y^{k+1}) + \lambda^{-p}\phi(w^k - w).
\end{align*}
Let $y \in \bR^m$. Plugging $w := (1 - \theta_k)y + \theta_k y^k$ in the identity above we obtain by concavity of $L(x^{k+1},\cdot)$, using $w^k - w = \theta_k y^k + (1 - \theta_k) y^0- (1 - \theta_k)y-\theta_k y^k=(1-\theta_k)(y^0 - y)$ and the homgeneity of $\phi$ with order $p+1$:
\begin{align*}
(1 - \theta_k) L(x^{k+1},y) + \theta_k L(x^{k+1}, y^k)&\leq L(x^{k+1}, y^{k+1}) + \lambda^{-p}(1-\theta_k)^{p+1}\phi(y^0 -y).
\end{align*}
Since $\theta_k = \tfrac{A_k}{A_{k+1}}$, and $1-\theta_k = \tfrac{a_{k+1}}{A_{k+1}}$ we obtain:
\begin{align*}
\tfrac{a_{k+1}}{A_{k+1}}L(x^{k+1},y) + \tfrac{A_k}{A_{k+1}} L(x^{k+1}, y^k)&\leq L(x^{k+1}, y^{k+1}) + \lambda^{-p}(\tfrac{a_{k+1}}{A_{k+1}})^{p+1}\phi(y^0 -y).
\end{align*}
Multiplication with $A_{k+1}$ yields:
\begin{align*}
a_{k+1} L(x^{k+1},y) + A_k L(x^{k+1}, y^k) &\leq A_{k+1} L(x^{k+1}, y^{k+1}) + \lambda^{-p}\tfrac{a_{k+1}^{p+1}}{A_{k+1}^p}\phi(y^0 -y).
\end{align*}
Using the relation $a_{k+1}=A_{K+1}-A_k$ the inequality can be rewritten:
\begin{align} \label{eq:ergodic_bound}
A_{k+1} \big(L(x^{k+1},y)-L(x^{k+1}, y^{k+1})\big) &\leq A_{k} \big(L(x^{k+1},y) - L(x^{k+1}, y^k)\big) + \lambda^{-p}\tfrac{a_{k+1}^{p+1}}{A_{k+1}^p}\phi(y^0 -y).
\end{align}
Thanks to the error bound for the $x$-update in \cref{alg:inexact_al_averaging} and using the convention $\varepsilon_{-1}:= L(x^0, y^0)-\inf L( \cdot, y^0)$ we have that for all $k \geq 0$, $L(x^k, y^k) \leq \inf L(\cdot, y^k) + \varepsilon_{k-1} \leq L(x^{k+1}, y^k) + \varepsilon_{k-1}.$
This yields
\begin{align}
L(x^{k+1},y)- L(x^k, y) + L(x^k, y)-L(x^k, y^k) + \varepsilon_{k-1}\geq L(x^{k+1},y)-L(x^{k+1}, y^k).
\end{align}
Plugging this inequality in \cref{eq:ergodic_bound} we obtain:
\begin{align*}
A_{k+1} \big(L(x^{k+1},y)-L(x^{k+1}, y^{k+1})\big) &\leq A_{k} \big(L(x^k, y)-L(x^k, y^k)\big) + A_{k} \big(L(x^{k+1},y)- L(x^k, y)\big) \\
&\qquad+ A_k \varepsilon_{k-1} 
+ \tfrac{a_{k+1}^{p+1}}{A_{k+1}^p}(\lambda \star\phi)(y^0 -y).
\end{align*}
Summing the inequality from $k=0$ to $k=K-1$ we obtain since $A_0 =0$ using the fact that $\sum_{k=1}^{K} \tfrac{a_{k}^{p+1}}{A_{k}^p} \leq (p+1)^{p+1}K$:
\begin{align*}
A_{K} L(x^{K},y)- A_kL(x^{K}, y^{K}) &\leq \sum_{k=0}^{K-1} A_{k} \big(L(x^{k+1},y)- L(x^k, y)\big) \\
&\qquad + \sum_{k=0}^{K-2} A_{k+1} \varepsilon_{k} + \lambda \star \phi(y^0 -y) (p+1)^{p+1}K.
\end{align*}
By reordering the sum using $a_{k+1} = A_{k+1}-A_k$ and $A_0=0$ we have the identity:
\begin{align*}
    &A_K L(x^K,y) -\sum_{k=0}^{K-1} A_{k} \big(L(x^{k+1},y)- L(x^k, y)\big) =\sum_{k=1}^{K}a_{k} L(x^{k}, y).
\end{align*}
This allows us to rewrite the previous inequality
\begin{align} \label{eq:ergodic_estimate_main}
\sum_{k=1}^{K}a_{k} L(x^{k}, y)-A_K L(x^{K}, y^{K}) &\leq \sum_{k=0}^{K-2} A_{k+1} \varepsilon_{k} + \lambda \star \phi(y^0 -y) (p+1)^{p+1}K.
\end{align}

By convexity of $L(\cdot, y)$ since $A_{K} = \sum_{j=1}^{K} a_j$ we have:
\begin{align}
A_K L(\breve x^K, y) \leq A_K \sum_{k=1}^{K} \tfrac{a_k}{\sum_{j=1}^K a_j} L(x^k, y) = \sum_{k=1}^{K} a_k L(x^k, y).
\end{align}
Plugging this relation in \cref{eq:ergodic_estimate_main} we can lower bound:
\begin{align*}
A_{K} \big(L(\breve x^{K}, y)-L(x^{K}, y^{K})\big) &\leq \sum_{k=0}^{K-2} A_{k+1} \varepsilon_{k} + \lambda \star \phi(y^0 -y) (p+1)^{p+1}K.
\end{align*}
By the $x$-update we have for any $x \in \bR^n$, $L(x^{K}, y^{K}) \leq \inf L(\cdot, y^K) + \varepsilon_{K-1} \leq L(x, y^K) + \varepsilon_{K-1}$
and hence we can lower bound since $A_{k+1}\varepsilon_k = c$
\begin{align*}
A_{K} \big(L(\breve x^{K}, y)-L(x, y^{K})\big) &\leq c K + \lambda \star \phi(y^0 -y) (p+1)^{p+1}K.
\end{align*}
Dividing by $A_k$ we obtain $L(\breve x^{K}, y)-L(x, y^{K}) \leq \tfrac{c + \lambda \star \phi(y^0 -y) (p+1)^{p+1}}{K^{p}}$ as claimed.
\ifx\ifsvjour\true
\qed
\fi
\end{proof}

%% file: proofs/ergodic_primal.tex
\begin{proof}
    From \cref{thm:primal-dual-gap} we have that for $K \geq 1$ and $x \in \bR^n, y\in \bR^m$
    \begin{align*}
        L(\breve x^{K}, y)-L(x, y^{K}) \leq \tfrac{c + \lambda \star \phi(y^0 -y) (p+1)^{p+1}}{K^{p}}
    \end{align*}
    and since for $x=x^\star$, $L(\breve x^{K}, y)-L(x^\star, y^{K}) = f(\breve x^{K}) - f(x^\star) + \langle A\breve x^{K} - b, y\rangle$ we obtain the following inequality for all $y \in \bR^m$:
    \begin{align} \label{eq:lagrangian_bound}
        f(\breve x^{K}) - f(x^\star) + \langle A\breve x^{K} - b, y\rangle \leq \tfrac{c + \lambda \star \phi(y^0 -y) (p+1)^{p+1}}{K^{p}}.
    \end{align}
    By taking $y = 0$ in the inequality above we obtain
    \begin{align} \label{eq:f_suboptimal_upper_bound}
         f(\breve x^{K}) - f(x^\star) \leq \tfrac{c + \lambda \star \phi(y^0) (p+1)^{p+1}}{K^{p}}.
    \end{align}
    Now, let $\delta = 2\|y^\star\|+1$ and take the maximum in \eqref{eq:lagrangian_bound} w.r.t. $y \in \delta \mathcal{B}$:
    \begin{align} \label{eq:summable_ineq_1}
        f(\breve x^{K}) - f(x^\star) + \delta \|A\breve x^{K} -b\|_* \leq \tfrac{c + \max_{y \in \delta \mathcal{B}}\lambda \star \phi(y^0-y)(p+1)^{p+1}}{K^{p}}.
    \end{align}
   Since $(x^\star, y^\star)$ is a saddle point, it holds that $L(x^\star, y) \leq L(x^\star, y^\star) \leq L(x, y^\star)$ for any $x\in \bR^n, y\in \bR^m$ and as such $L(x,y^\star) - L(x^\star, y) \geq 0$. Choose $x = \breve x^{K}$ and $y=0$. Then we obtain
   \begin{align} \nonumber
       0 &\leq L(\breve x^{K},y^\star) - L(x^\star, 0) \\
       &= f(\breve x^{K}) - f(x^\star) + \langle A \breve x^{K} - b, y^\star \rangle  \nonumber \\       
       & \leq f(\breve x^{K}) -f(x^\star) + \|y^\star\| \|A\breve x^{K} -b\|_*, 
       \label{eq:summable_ineq_2}
   \end{align}
   where the last inequality follows by applying H\"older's inequality. Summing now \eqref{eq:summable_ineq_1} and \eqref{eq:summable_ineq_2}, we get since $\delta > \|y^\star\|$
   \begin{align}
       \|A\breve x^{K} -b\|_* \leq (\delta - \|y^\star\|)\|A\breve x^{K} -b\|_* \leq \tfrac{c + \max_{y \in \delta \mathcal{B}}\lambda \star \phi(y^0-y)(p+1)^{p+1}}{K^{p}}.
   \end{align}
   From the inequality above together with \eqref{eq:summable_ineq_2} we obtain:
   \begin{align} \nonumber
       f(x^\star) - f(\breve x^{K}) \leq \|y^\star\| \|A\breve x^{K} -b\|_* 
       &\leq (\delta - \|y^\star\|)\|A\breve x^{K} -b\|_*
       \\ 
       &\leq \tfrac{c + \max_{y \in \delta \mathcal{B}}\lambda \star \phi(y^0-y)(p+1)^{p+1}}{K^{p}}, 
   \end{align}
   which along with \eqref{eq:f_suboptimal_upper_bound} implies the claimed result.
\ifx\ifsvjour\true
\qed
\fi
\end{proof}

%% file: additional_tables.tex
In addition to the median number of BFGS iterations required to solve the LPs reported in \cref{table:lps}, \cref{table:lps-p95} provides the corresponding $95^\text{th}$ percentiles (P95), which qualitatively confirm the previously reported findings.
In a similar way, \cref{table:quadratics-p95} provides the corresponding $95^\text{th}$ percentiles (P95) for the QP experiments reported in \cref{table:quadratics}.

\begin{table}
    \caption{
        P95 of the total number of BFGS iterations required to solve the LP, computed over $N=20$ random realizations.
    }
    \label{table:lps-p95}
    \centering
    \begin{adjustbox}{width=\textwidth}
    \setlength\extrarowheight{5pt}
    
    \pgfplotstabletypeset[%
        begin table={\begin{tabular}[t]},
        every head row/.style={
            before row={%
              \hline
              \vphantom{$q = 1$}\\
              \hline
              \vphantom{$q = 1$}\\
              \hline
            },
            after row/.add={}{\hline},
        },
        header=true,
        col sep=&,
        row sep=\\,
        string type,
        columns/{$(m, n)$}/.style ={column name={$(m, n)$}, column type={|c}},
        every row no 7/.style={after row=\hline},
    ]{
        \\
        $(m, n)$\\
        (200, 100)\\
        (400, 200)\\
        (600, 300)\\
        (300, 100)\\
        (600, 200)\\
        (400, 100)\\
        (500, 100)\\
        (600, 100)\\
    }%
    \begin{filecontents*}{data/lp_p95.csv}
2.684900000000002819e+03 3.251250000000000000e+03 1.794200000000002547e+03 2.527650000000000546e+03 1.873200000000008458e+03 2.345650000000001455e+03 1.426100000000001728e+03 2.476300000000000182e+03
3.994800000000000637e+03 3.968800000000000182e+03 2.699850000000000819e+03 3.656050000000000182e+03 3.117300000000001546e+03 3.572500000000000000e+03 2.470150000000000546e+03 4.033500000000000000e+03
5.511550000000000182e+03 4.915000000000000000e+03 5.747800000000004729e+03 4.415600000000000364e+03 7.440650000000012369e+03 4.700150000000000546e+03 3.808650000000000546e+03 4.810550000000000182e+03
1.136950000000000045e+03 1.837400000000000546e+03 1.182850000000000136e+03 1.065450000000000045e+03 1.566600000000000136e+03 1.046550000000000182e+03 1.092550000000000182e+03 1.322250000000000455e+03
2.037700000000000500e+03 2.520750000000000000e+03 2.457000000000000000e+03 1.990299999999999955e+03 3.386250000000000000e+03 1.774350000000000136e+03 2.369000000000000455e+03 1.936250000000000000e+03
1.585700000000000500e+03 1.650200000000000045e+03 1.567750000000000909e+03 1.158099999999999909e+03 2.716100000000002638e+03 1.035000000000000000e+03 1.609000000000000455e+03 1.103799999999999955e+03
2.709750000000003638e+03 1.769450000000000500e+03 3.716950000000000273e+03 1.375000000000000227e+03 7.270350000000003092e+03 1.625500000000000909e+03 2.744650000000001455e+03 1.237799999999999955e+03
2.536349999999999909e+03 1.659000000000000000e+03 2.750449999999999818e+03 1.327200000000000045e+03 4.866250000000003638e+03 1.303650000000000091e+03 2.212950000000000273e+03 1.122250000000000000e+03
\end{filecontents*}
\pgfplotstabletypeset[%
        begin table={\begin{tabular}[t]},
        every head row/.style={
        before row={%
          \hline
          \multicolumn{4}{|c|}{Classical ALM} & \multicolumn{4}{|c|}{Power ALM}\\
          \hline
          \multicolumn{2}{|c|}{$q = 1$, fixed $\lambda$} & \multicolumn{2}{c|}{$q = 1$, adaptive $\lambda$} & \multicolumn{2}{c|}{$q \neq 1, \lambda = 10^2$} & \multicolumn{2}{c|}{$q \neq 1, \lambda = 10^3$}\\
          \hline
        },
        after row/.add={}{\hline},
        },
        header=true,
       precision=1,
       columns/0/.style ={column name={$\lambda = 10^3$}, column type={|l}},
       columns/1/.style ={column name={$\lambda = 10^4$}, column type={r|}},
       columns/2/.style ={column name={$\lambda_0 = 10^2$}, column type={l}},
       columns/3/.style ={column name={$\lambda_0 = 10^3$}, column type={r|}},
       columns/4/.style ={column name={$q = 0.9$}, column type={l}},
       columns/5/.style ={column name={$q = 0.8$}, column type={r|}},
       columns/6/.style ={column name={$q = 0.9$}, column type={l}},
       columns/7/.style ={column name={$q = 0.8$}, column type={r|}},
        every row no 7/.style={after row=\hline},
       every row 0 column 6/.style={
            postproc cell content/.style={
              @cell content/.add={$\bf}{$}
            }
        },
       every row 1 column 6/.style={
            postproc cell content/.style={
              @cell content/.add={$\bf}{$}
            }
        },
       every row 2 column 6/.style={
            postproc cell content/.style={
              @cell content/.add={$\bf}{$}
            }
        },
       every row 3 column 5/.style={
            postproc cell content/.style={
              @cell content/.add={$\bf}{$}
            }
        },
       every row 4 column 5/.style={
            postproc cell content/.style={
              @cell content/.add={$\bf}{$}
            }
        },
       every row 5 column 5/.style={
            postproc cell content/.style={
              @cell content/.add={$\bf}{$}
            }
        },
       every row 6 column 7/.style={
            postproc cell content/.style={
              @cell content/.add={$\bf}{$}
            }
        },
       every row 7 column 7/.style={
            postproc cell content/.style={
              @cell content/.add={$\bf}{$}
            }
        },
    ]{data/lp_p95.csv}
    \end{adjustbox}
    
\end{table}

\begin{table}
    \caption{
        P95 of the total number of iterations of \cite{li2024simpleuniformlyoptimalmethod} required to solve the QP, over $N=20$ random realizations.
    }
    \label{table:quadratics-p95}
    \centering
    \begin{adjustbox}{width=\textwidth}
    \setlength\extrarowheight{5pt}
    
    \pgfplotstabletypeset[%
        begin table={\begin{tabular}[t]},
        every head row/.style={
            before row={%
              \hline
              \vphantom{$q = 1$}\\
              \hline
              \vphantom{$q = 1$}\\
              \hline
            },
            after row/.add={}{\hline},
        },
        header=true,
        col sep=&,
        row sep=\\,
        string type,
        columns/{$(m, n)$}/.style ={column name={$(m, n)$}, column type={|c}},
        every row no 9/.style={after row=\hline},
    ]{
        \\
        $(m, n)$\\
        (200, 400)\\
        (250, 500)\\
        (300, 600)\\
        (350, 700)\\
        (400, 800)\\
        (450, 900)\\
        (150, 450)\\
        (200, 600)\\
        (250, 750)\\
        (300, 900)\\
    }%
    \begin{filecontents*}{data/convex_quadratic_v2_p95.csv}
4.913350000000000364e+03 5.276300000000000182e+03 1.214755000000000109e+04 6.394699999999999818e+03 4.808300000000000182e+03 6.496550000000001091e+03 4.539600000000000364e+03 4.822899999999999636e+03 4.841900000000000546e+03
4.754949999999999818e+03 5.767600000000000364e+03 1.339390000000000146e+04 6.964950000000000728e+03 4.995550000000000182e+03 5.643900000000000546e+03 4.556300000000000182e+03 4.408500000000000000e+03 5.433699999999999818e+03
5.439300000000000182e+03 5.803650000000000546e+03 1.592755000000000109e+04 6.729500000000000909e+03 5.962449999999999818e+03 6.746800000000001091e+03 5.363550000000000182e+03 7.303750000000000000e+03 5.221449999999999818e+03
5.517050000000000182e+03 7.438300000000002910e+03 1.754565000000000146e+04 6.690700000000000728e+03 9.008950000000002547e+03 6.063150000000000546e+03 5.493300000000001091e+03 4.958100000000002183e+03 5.162100000000001273e+03
6.632250000000000909e+03 7.155100000000000364e+03 1.959560000000000218e+04 7.924450000000002547e+03 6.124500000000000909e+03 7.555900000000000546e+03 5.015750000000000909e+03 5.285149999999999636e+03 4.717650000000000546e+03
5.704800000000000182e+03 6.004950000000000728e+03 2.553655000000000291e+04 6.811800000000000182e+03 5.788949999999999818e+03 8.132350000000001273e+03 5.043650000000000546e+03 5.251750000000000909e+03 4.709550000000000182e+03
6.718350000000000364e+03 6.564650000000000546e+03 1.464480000000000291e+04 8.025199999999999818e+03 6.838449999999999818e+03 6.405449999999999818e+03 5.264200000000000728e+03 5.761700000000001637e+03 5.465700000000000728e+03
6.843699999999999818e+03 6.833350000000002183e+03 1.761370000000000073e+04 8.202700000000000728e+03 6.133149999999999636e+03 9.162550000000012005e+03 6.058700000000000728e+03 7.147350000000000364e+03 5.819100000000000364e+03
8.676750000000001819e+03 7.608750000000000909e+03 1.970905000000001019e+04 8.217350000000004002e+03 8.162700000000000728e+03 8.732900000000001455e+03 8.206149999999999636e+03 1.080055000000000109e+04 6.195300000000001091e+03
9.759400000000001455e+03 1.046515000000000146e+04 3.029030000000000291e+04 9.697300000000001091e+03 6.582500000000001819e+03 1.188375000000000000e+04 8.175100000000000364e+03 1.466205000000000109e+04 6.541200000000002547e+03
\end{filecontents*}
\pgfplotstabletypeset[%
        begin table={\begin{tabular}[t]},
        every head row/.style={
        before row={%
          \hline
          \multicolumn{6}{|c|}{Classical ALM} & \multicolumn{3}{|c|}{Power ALM}\\
          \hline
          \multicolumn{3}{|c|}{$q = 1$, fixed $\lambda$} & \multicolumn{3}{c|}{$q = 1$, adaptive $\lambda$} & \multicolumn{3}{c|}{$q \neq 1, \lambda = 0.1$}\\
          \hline
        },
        after row/.add={}{\hline},
        },
        header=true,
       precision=1,
       columns/0/.style ={column name={$\lambda = 0.1$}, column type={|l}},
       columns/1/.style ={column name={$\lambda = 1$}, column type={c}},
       columns/2/.style ={column name={$\lambda = 10$}, column type={r|}},
       columns/3/.style ={column name={$\lambda_0 = 0.01$}, column type={l}},
       columns/4/.style ={column name={$\lambda_0 = 0.1$}, column type={c}},
       columns/5/.style ={column name={$\lambda_0 = 1$}, column type={r|}},
       columns/6/.style ={column name={$q = 0.9$}, column type={l}},
       columns/7/.style ={column name={$q = 0.8$}, column type={c}},
       columns/8/.style ={column name={$q = 0.7$}, column type={r|}},
        every row no 9/.style={after row=\hline},
       every row 0 column 6/.style={
            postproc cell content/.style={
              @cell content/.add={$\bf}{$}
            }
        },
       every row 1 column 7/.style={
            postproc cell content/.style={
              @cell content/.add={$\bf}{$}
            }
        },
       every row 2 column 8/.style={
            postproc cell content/.style={
              @cell content/.add={$\bf}{$}
            }
        },
       every row 3 column 7/.style={
            postproc cell content/.style={
              @cell content/.add={$\bf}{$}
            }
        },
       every row 4 column 8/.style={
            postproc cell content/.style={
              @cell content/.add={$\bf}{$}
            }
        },
       every row 5 column 8/.style={
            postproc cell content/.style={
              @cell content/.add={$\bf}{$}
            }
        },
       every row 6 column 6/.style={
            postproc cell content/.style={
              @cell content/.add={$\bf}{$}
            }
        },
       every row 7 column 8/.style={
            postproc cell content/.style={
              @cell content/.add={$\bf}{$}
            }
        },
       every row 8 column 8/.style={
            postproc cell content/.style={
              @cell content/.add={$\bf}{$}
            }
        },
       every row 9 column 8/.style={
            postproc cell content/.style={
              @cell content/.add={$\bf}{$}
            }
        },
       every row 10 column 7/.style={
            postproc cell content/.style={
              @cell content/.add={$\bf}{$}
            }
        },
       every row 11 column 7/.style={
            postproc cell content/.style={
              @cell content/.add={$\bf}{$}
            }
        },
       every row 12 column 7/.style={
            postproc cell content/.style={
              @cell content/.add={$\bf}{$}
            }
        },
       every row 13 column 7/.style={
            postproc cell content/.style={
              @cell content/.add={$\bf}{$}
            }
        },
       every row 14 column 7/.style={
            postproc cell content/.style={
              @cell content/.add={$\bf}{$}
            }
        },
       every row 15 column 7/.style={
            postproc cell content/.style={
              @cell content/.add={$\bf}{$}
            }
        },
       every row 16 column 7/.style={
            postproc cell content/.style={
              @cell content/.add={$\bf}{$}
            }
        },
    ]{data/convex_quadratic_v2_p95.csv}
    \end{adjustbox}
\end{table}

%% file: proofs/alm.tex
\begin{proof}
``\labelcref{thm:errors_alm:al}'':
Since $y^+ \in \argmin_{\eta \in \bR^m} \lambda \star \phi(y - \eta)-L(x^+, \cdot)$ we have that
$$
v = \nabla \phi(\lambda^{-1}(y - y^+)) \in \partial_y(-L)(x^+, y^+).
$$
By concavity of $L(x, \cdot)$ we have that for all $y \in \bR^m$
$$
\langle v, y^+-y \rangle \geq L(x^+, y)-L(x^+, y^+) \geq \varrho(y)-L(x^+, y^+),
$$
where the second inequality follows by the definition of the dual function $\varrho$. Since $L(x^+, y^+) \leq \inf L(\cdot, y^+) + \varepsilon$ we can further bound
$
\langle v, y^+-y \rangle \geq \varrho(y)-\varrho(y^+) - \varepsilon.
$
Rearranging this yields
$
(-\varrho)(y) \geq (-\varrho)(y^+)+\langle v,y- y^+ \rangle - \varepsilon,
$
for any $y \in \bR^m$ and hence $v \in \partial_{\varepsilon}(-\varrho)(y^+)$, i.e. $v \in \partial_{\varepsilon}\psi(y^+)$.

``\labelcref{thm:errors_prox_alm:bound_alm}'': Using the concavity of the augmented Lagrangian function $L_\lambda(x^{k+1}, \cdot)$, we get
\begin{align} \nonumber
L_\lambda (x^+,y) + \langle \nabla_y L_\lambda(x^+, y), \eta - y \rangle &\geq L_\lambda (x^+, \eta) \geq \inf_{x \in \bR^n} L_\lambda(x, \eta) \\ \nonumber
&=\inf_{x \in \bR^n} \max_{y \in \bR^m}~L(x,y) -\lambda \star \phi(\eta - y) \\ \nonumber
&=\max_{y \in \bR^m} \inf_{x \in \bR^n} ~L(x,y) -\lambda \star \phi(\eta - y) \\ \nonumber
&=\max_{y \in \bR^m} \varrho(y) -\lambda \star\phi(\eta - y) \\
&\geq \varrho(\aprox{\lambda\star\phi}{\psi}(y)) - \lambda \star\phi(\eta -\aprox{\lambda\star\phi}{\psi}(y)), \label{eq:aug_lag_bound}
\end{align}
where we have used the definitions of the Lagrangian dual function, the properties of $\inf$ and $\max$ and strong duality.
The same reasoning also yields:
\begin{align} \nonumber
\inf_{x \in \bR^n} L_\lambda(x,y) &=\inf_{x \in \bR^n} \max_{\eta \in \bR^m}~L(x,\eta) -\lambda \star\phi(y - \eta) =\max_{\eta \in \bR^m} \varrho(\eta)-\lambda \star\phi(y - \eta) \\
&=\varrho(\aprox{\lambda\star\phi}{\psi}(y))-\lambda \star\phi(y - \aprox{\lambda\star\phi}{\psi}(y)). \label{eq:min_aug_lag_bound}
\end{align}
Now, from the optimality condition in  \cref{eq:general_inexact_dual}, we have that $\lambda^{-p}\nabla \phi(y - y^+) = \nabla_y L_\lambda(x^+,y)$. Therefore, by combining \cref{eq:aug_lag_bound} and \cref{eq:min_aug_lag_bound}, we obtain the following result: 
\begin{align*}
&L_\lambda(x^+, y) - \inf_x L_\lambda(x, y) \\ 
&\qquad \geq  \lambda \star\phi(y - \aprox{\lambda\star\phi}{\psi}(y)) - \lambda \star\phi(\eta - \aprox{\lambda\star\phi}{\psi}(y)) + \lambda^{-p}\langle \nabla \phi(y - y^+), \eta - y \rangle,
\end{align*}
which holds for all $\eta \in \bR^m$ so that we can replace $\eta$ with its maximum. This is attained if $\lambda^{-p}\nabla \phi(\eta - \aprox{\lambda\star\phi}{\psi}(y)) = \lambda^{-p}\nabla \phi(y-y^+)$, i.e., $\eta = y-y^+ + \aprox{\lambda\star\phi}{\psi}(y)$.
Substituting this result in the inequality above, we finally arrive at the claimed inequality
\begin{align*}
L_\lambda(x^+, y) - \inf_x L_\lambda(x, y) &\geq \lambda \star\phi(y - \aprox{\lambda\star\phi}{\psi}(y)) - \lambda \star \phi(y-y^+)
\\ 
&\qquad - \lambda^{-p}\langle \nabla \phi(y - y^+), y-\aprox{\lambda\star\phi}{\psi}(y) - (y- y^+) \rangle
\\
&\geq \lambda^{-p} \left(\tfrac{1}{2}\right)^{p-1}\tfrac{1}{p+1}\|y^+ - \aprox{\lambda\star\phi}{\psi}(y)\|^{p+1},
\end{align*}
where the final inequality follows from \cref{thm:uniform_convexity}.

``\labelcref{thm:errors_prox_alm:grad_alm}'': We have that $e \in \partial_x L(x^+, y^+)$ and $0 \in \partial_y (-L)(x^+, y^+) - \nabla \phi(\lambda^{-1}(y - y^+))$ if and only if
$$
(x^+, y^+) \in \argminimax_{x \in \bR^n, \eta \in \bR^m} ~L(x, \eta) - \langle e, x \rangle - \lambda \star \phi(y-\eta),
$$
and hence
\begin{align} \label{eq:augmented_lagrangian_gradient}
e \in \partial_x L(x^+, y^+) \Leftrightarrow x^+ \in \argmin_{x \in \bR^n} ~L_{\lambda}(x, y) - \langle e, x \rangle \Leftrightarrow e \in \partial_x L_{\lambda}(x^+, y).
\end{align}
By convexity of $L(\cdot, y^+)$ we have for all $x \in \bR^m$ that
$
L(x, y^+) \geq L(x^+, y^+) + \langle e, x - x^+ \rangle,
$
and hence
$
\inf L(\cdot, y^+) \geq L(x^+, y^+) - \|e\| \cdot \inf_{x \in \dom f} \|x - x^+\|
$
implying that
$
L(x^+, y^+) \leq \inf L(\cdot, y^+) + \|e\| D \leq \inf L(\cdot, y^+) +\varepsilon.$

For the second part note that by convexity of $L_\lambda(\cdot, y)$, we have that
$
    \inf L_\lambda(\cdot, y) - L_\lambda(x^+, y) \geq \langle e, x-x^+\rangle,
$
or by rearranging and using the Cauchy--Schwarz inequality $
    L_\lambda(x^+, y) - \inf L_\lambda(\cdot, y) \leq \|e\| D \leq \varepsilon,$
which is the desired result.
\ifx\ifsvjour\true
\qed
\fi
\end{proof}

%% file: proofs/global_linear.tex
\begin{proof}
    ``\labelcref{thm:linear_eq:feasibility}'': The proof is based on the equivalence with the higher-order proximal-point method applied to the negative dual function $\psi = -\varrho$. Invoking \cref{thm:decrease_inexact_ppa} for points $y = y^k$ and $w = y^k$:
    \begin{align*}
        \psi(y^{k+1}) \leq \psi(y^k) - \lambda^{-p} (1 + \tfrac{1}{2^{p-1}})\phi(y^{k+1}-y^k) + \varepsilon_k.
    \end{align*}
    By rearranging and using the fact that $\psi(y^{k+1}) \geq \inf \psi$ we obtain the following inequality:
    \begin{align*}
        \lambda^{-p}(1 + \tfrac{1}{2^{p-1}})\phi(y^{k+1}-y^k) \leq \psi(y^k) - \inf \psi + \varepsilon_k.
    \end{align*}
    Substituting \cref{eq:dual_update_linear_equality} into the inequality above and using the fact that $\phi(\nabla \phi^*(x)) = \|x\|_*^q$, we have that
    \begin{align*}
        \|Ax^{k+1}-b\|_*^{\tfrac{p+1}{p}} \leq (\lambda(1 + \tfrac{1}{2^{p-1}}))^{-1}(\psi(y^k) - \inf \psi + \varepsilon_k).
    \end{align*}
    Using now \cref{thm:inexact_ppa_rate}, we obtain
    \begin{align*}
        \|Ax^{k+1}-b\|_* \leq (\sigma \eta_k)^{\tfrac{p}{p+1}},
    \end{align*}
    for $\eta_k = \tfrac{\lambda^{-p}D_\beta^{p+1} (p+1)^{p} + c}{k^p}$ and $\sigma = (\lambda(1 + \tfrac{1}{2^{p-1}}))^{-1}$, which implies the claimed result.

    ``\labelcref{thm:linear_eq:f_bound}'': To begin with, using the fact that strong duality holds, choosing a saddle point $(x^\star, y^\star)$ and using H\"older's inequality, we have that
    \begin{align*}
        f(x^\star) &\leq L_\lambda(x^{k+1}, y^\star) \\
        &\leq f(x^{k+1}) + \|y^\star\|\|Ax^{k+1}-b\|_* + \lambda^{-p}\phi^*(Ax^{k+1}-b),
    \end{align*}
    which using the previous result readily implies that
    \begin{align} \label{eq:f_lower_bound}
        f(x^{k+1}) - f(x^\star) \geq -\|y^\star\| (\sigma \eta_k)^{\tfrac{p}{p+1}} - \tfrac{\lambda^{-p}}{q+1} \sigma \eta_k
    \end{align}
    Now, in order to prove the other side of the inequality for the absolute value, using the error condition of the update and the relation $\inf L(\cdot, y^{k+1}) = \varrho(y^{k+1})$ we have
    \begin{align*}
        f(x^{k+1}) + \langle y^{k+1},Ax^{k+1}-b \rangle = L(x^{k+1},y^{k+1}) = \inf L(\cdot,y^{k+1}) + \varepsilon_k = \varrho(y^{k+1}) + \varepsilon_k,
    \end{align*}
    and by using the bound $\|y^{k+1}-y^\star\| \leq D_\beta$, the triangle inequality, as well as the fact that $\varrho(y^{k+1}) \leq f(x^\star)$, we can further write:
    \begin{align} \label{eq:f_upper_bound}
        f(x^{k+1}) - f(x^\star) \leq (\|y^\star\|+D_\beta)(\sigma \eta_k)^{\tfrac{p}{p+1}} + \varepsilon_k
    \end{align}
    The claimed result now follows by combining \cref{eq:f_lower_bound} and \cref{eq:f_upper_bound}.
\ifx\ifsvjour\true
\qed
\fi
\end{proof}